\documentclass[12pt,leqno]{article}
\usepackage{amsmath,amssymb,amsthm,graphicx,MnSymbol,mathtools,url}

\DeclareMathOperator{\graph}{graph}
\DeclareMathOperator{\projection}{proj}
\DeclareMathOperator{\spt}{spt}
\DeclareMathOperator{\divergence}{div}
\DeclareMathOperator{\reg}{reg}
\DeclareMathOperator{\sing}{sing}
\DeclareMathOperator{\dist}{dist}
\DeclareMathOperator{\clos}{clos}
\DeclareMathOperator{\op}{p}
\DeclareMathOperator{\trace}{trace}

\def\res{\hbox{ {\vrule height .22cm}{\leaders\hrule\hskip.2cm} }}

\newcommand{\C}{\mathbb{C}}

\newcommand{\D}{\overline{D}}
\newcommand{\uD}{\underline{D}}

\newcommand{\x}{\mathbf{x}}
\newcommand{\z}{\mathbf{z}}
\newcommand{\E}{\mathbf{E}}
\newcommand{\I}{\mathbf{I}}
\newcommand{\N}{\mathbf{N}}
\newcommand{\R}{\mathbf{R}}
\newcommand{\TI}{\mathbf{TI}}

\newcommand{\CC}{\mathcal{C}}
\newcommand{\HH}{\mathcal{H}}

\newcommand{\proj}[2]{\projection_{#1} #2}
\newcommand{\dive}[2]{\divergence_{#1} #2}
\newcommand{\gph}[2]{\graph_{#1} #2}
\newcommand{\bop}{\boldsymbol{\op}}

\numberwithin{equation}{section}
\newtheorem{theorem}[equation]{Theorem}
\newtheorem{lemma}[equation]{Lemma}
\newtheorem{definition}[equation]{Definition}

\begin{document}
\begin{flushleft}

TITLE: Co-dimension one area-minimizing currents with $C^{1,\alpha}$ tangentially immersed boundary having Lipschitz co-oriented mean curvature.

\medskip

AUTHOR: Leobardo Rosales, Keimyung University

\medskip

ABSTRACT: We study $n$-dimensional area-minimizing currents $T$ in $\R^{n+1},$ with boundary $\partial T$ satisfying two properties: $\partial T$ is locally a finite sum of $(n-1)$-dimensional $C^{1,\alpha}$ orientable submanifolds which only meet tangentially and with same orientation, for some $\alpha \in (0,1]$; $\partial T$ has mean curvature $=h \nu_{T}$ where $h$ is a Lipschitz scalar-valued function and $\nu_{T}$ is the generalized outward pointing normal of $\partial T$ with respect to $T.$ We give a partial boundary regularity result for such currents $T.$ We show that near any point $x$ in the support of $\partial T,$ either the support of $T$ has very uncontrolled structure, or the support of $T$ near $x$ is the finite union of orientable $C^{1,\alpha}$ hypersurfaces-with-boundary with disjoint interiors and common boundary points only along the support of $\partial T.$

\medskip


\medskip


\section{Introduction}

This work continues the topic introduced and studied in \cite{R18}. We consider $n$-dimensional area-minimizing locally rectifiable currents $T$ in $\R^{n+1}$ with boundary $\partial T$ satisfying two properties as follows. First, we suppose there is an $\alpha \in (0,1]$ so that $\partial T$ can locally be written as a finite sum of $C^{1,\alpha}$ orientable $(n-1)$-dimensional (embedded) submanifolds which meet only tangentially with equal orientation; we thus say that $T$ has \emph{$C^{1,\alpha}$ tangentially immersed boundary}, see Definition \ref{immersedboundary}. Second, we suppose $T$ has \emph{Lipschitz co-oriented mean curvature}, see Definition \ref{cmcboundary}; this means that $\partial T$ has generalized mean curvature $H_{\partial T} = h \nu_{T}$ where $h$ is Lipschitz and $\nu_{T}$ is the generalized outward pointing unit normal of $\partial T$ with respect to $T$ (see Lemma 3.1 of \cite{B77} and (2.9) of \cite{E89} for the existence of $\nu_{T}$).

\bigskip

The main result we aim to prove we heuristically state as follows:

\bigskip

{\bf Theorem \ref{main}} \emph{Suppose $T$ is an $n$-dimensional area-minimizing locally rectifiable current in $\R^{n+1}$ with $C^{1,\alpha}$ tangentially immersed boundary $\partial T$ for some $\alpha \in (0,1],$ where $\partial T$ has Lipschitz co-oriented mean curvature. Also suppose $x$ is a singular point of $\partial T,$ and near $x$ the support of $T$ equals a finite union of orientable $C^{1,\alpha}$ hypersurfaces-with-boundary. Then near $x,$ the support of $T$ equals a finite union of orientable $C^{1,\alpha}$ hypersurfaces-with-boundary, which pairwise meet only along common boundary points.}

\bigskip

\begin{figure}[] \begin{center} \includegraphics[height=1.5in]{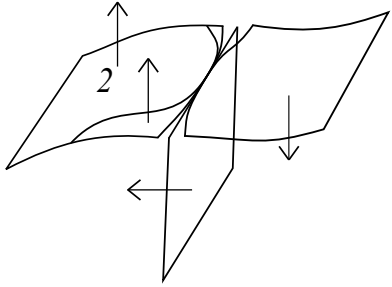} \end{center} \caption{An illustration of the conclusions of Theorem \ref{main}.} \label{mainexample} \end{figure}

\bigskip

Figure \ref{mainexample} shows a current satisfying the conclusion of Theorem \ref{main}. There, $T$ is given by integrating over three surfaces-with-boundary, with orientation as illustrated. Over two of the surfaces, $T$ is given multiplicity one, whereas in the remaining surface $T$ is given multiplicity two over a region of that surface as labeled. Thus, $T$ has tangentially immersed boundary consisting of four curves. The better way to understand Theorem \ref{main} is through the contrapositive: at any point $x$ in the support of $\partial T$ near which the support of $T$ is not as in Figure \ref{mainexample}, $T$ near $x$ must have complicated structure; perhaps for example having infinite topology at $x.$

\bigskip

Previously, Theorem \ref{main} was shown in case $\alpha = 1,$ see Theorem 5.3 of \cite{R18}. In case $n=2,$ Theorem 6.5 of \cite{R18} shows that Theorem \ref{main} holds for general $\alpha \in (0,1].$ Hence, to prove Theorem \ref{main} we only need to consider the case $n \geq 3,$ with in particular $\alpha \in (0,1).$

\subsection{Proof of Theorem \ref{main}}

The proof of Theorem \ref{main} follows the general strategy set by the proof of Theorem 5.3 of \cite{R18}. The key is to consider \emph{half-regular} singular points of $\partial T.$ These points are defined in Lemma \ref{halfregular}, which we roughly describe:

\bigskip

{\bf Lemma \ref{halfregular}:} \emph{Suppose $T$ is an $n$-dimensional area-minimizing locally rectifiable current in $\R^{n+1}$ with $C^{1,\alpha}$ tangentially immersed boundary $\partial T$ for some $\alpha \in (0,1].$ If $x$ is a singular point of $\partial T,$ then for any $\rho>0$ there exists $\x$ in the singular set of $\partial T$ with $|\x-x|<\rho$ and a non-empty open set $U \subset \R^{n+1}$ so that $\x \in \partial U$ and the support of $\partial T$ in $U$ is a union of \emph{disjoint} non-empty $(n-1)$-dimensional submanifolds; we call such an $\x$ a \emph{half-regular} point of $T.$}

\bigskip

This is Lemma 3.8 of \cite{R18}, but we give it here for convenience. A version of Lemma \ref{halfregular} also appears as Lemma 1 of \cite{R16} in the context of two-dimensional solutions to the $c$-Plateau problem in space. To prove Theorem \ref{main}, we must show the following asymptotic description of $T$ at half-regular points.

\bigskip

{\bf Lemma \ref{halfregulartangentcones}:} \emph{Suppose $T$ is an $n$-dimensional area-minimizing locally rectifiable current in $\R^{n+1}$ with $C^{1,\alpha}$ tangentially immersed boundary $\partial T$ for some $\alpha \in (0,1],$ where $\partial T$ has Lipschitz co-oriented mean curvature. If $\x$ is a half-regular point of $T,$ then every tangent cone of $T$ at $\x$ is a \emph{sum of half-hyperplanes with constant orientation after rotation} (see Definition \ref{cones}).}

\bigskip

Succinctly, Definition \ref{cones} defines a \emph{sum of half-hyperplanes with constant orientation after rotation} to be a current given by integrating over a union of half-hyperplanes meeting along a common $(n-1)$-dimensional subspace of $\R^{n+1},$ where the half-hyperplanes are oriented analogously to Figure \ref{mainexample}.

\bigskip

The proof of Lemma \ref{halfregulartangentcones} follows through naturally. Suppose for contradiction that (after translation) $0 \in \spt \partial T$ is a half-regular point of $T$ with a tangent cone which is not a sum of half-hyperplanes with constant orientation after rotation. Then we will see that Theorem \ref{hyperplanetangentconegraph} implies that (after rotation) $T$ near $0$ is supported in the graph of a function $u$ defined off $\R^{n}.$ By definition of half-regular points, we conclude there is a $C^{1,\alpha}$ subset $\Omega \subset \R^{n}$ with $0 \in \partial \Omega$ so that $\spt \partial T \cap (\Omega \times \R)$ is a union of \emph{disjoint} non-empty $(n-1)$-dimensional submanifolds contained in the graph of $u.$ The proof of Theorem \ref{hyperplanetangentconeregularity} then concludes by applying a generalization of the Hopf boundary point lemma to $\partial T$ at $0,$ using that $\partial T$ has Lipschitz co-oriented mean curvature.

\bigskip

Given Lemma \ref{halfregulartangentcones}, then the proof of Theorem \ref{main} follows word-for-word as in the proof of Theorem 5.3 of \cite{R18}, which shows Theorem \ref{main} in case $\alpha = 1.$ Indeed, the proof of Theorem 5.3 of \cite{R18} is very geometric, given Lemma 5.1 of \cite{R18} which concludes the same asymptotic result at half-regular points as Lemma \ref{halfregulartangentcones} but in case $\alpha = 1.$ Thus, the majority of our work involves proving Lemma \ref{halfregulartangentcones}.

\bigskip

The proof of Lemma 5.1 of \cite{R18} concludes by applying the usual Hopf boundary point lemma (see for example Lemma 3.4 of \cite{GT83}), since $\alpha=1$ in that case. Here, to prove Lemma \ref{halfregulartangentcones} for general $\alpha \in (0,1],$ we must apply a more general version of the Hopf boundary point lemma recently proved by the author in \cite{R17}. We state the needed version as Lemma \ref{appendixhopflemma} for convenience in the Appendix. Succinctly, as needed here, Lemma \ref{appendixhopflemma} states that the Hopf boundary point lemma holds for $C^{1,\alpha}$ weak solutions $s$ of equations of the form
$$\sum_{i,j=1}^{n-1} \D_{i} \left( a^{ij} \D_{j}s \right) + \sum_{i=1}^{n-1} c^{i} \D_{i}s + ds = 0$$
over a $C^{1,\alpha}$ domain, where the $a^{ij}$ are (as usual) $C^{0,\alpha}$ and uniformly elliptic, $c^{i}$ are bounded, but we can assume $d \in L^{q}$ for some $q>n-1.$ More generally, \cite{R17} shows the Hopf boundary point lemma holds assuming the lower-order coefficients are merely in a \emph{Morrey space}.

\subsection{Future Work}

A tempting conjecture to make is that the conclusion of Theorem \ref{main} holds without the initial regularity assumption. 

\bigskip

{\bf Conjecture:} \emph{Suppose $T$ is an $n$-dimensional area-minimizing locally rectifiable current in $\R^{n+1}$ with $C^{1,\alpha}$ tangentially immersed boundary $\partial T$ for some $\alpha \in (0,1],$ where $\partial T$ has Lipschitz co-oriented mean curvature. If $x$ is a singular point of $\partial T,$ then near $x$ the support of $T$ equals a finite union of orientable $C^{1,\alpha}$ hypersurfaces-with-boundary, which pairwise meet only along common boundary points.}

\bigskip

However, much work is needed in this direction. It may at least be possible to show that $T$ has unique tangent cone at every singular point $x$ of $\partial T.$ Contrarily, there is no hope to extend the present results in the case of general $n$-dimensional area-minimizing currents in $\R^{n+k}$; this is discussed in the last paragraph of Section 1.3 of \cite{R18}.

\subsection{Outline}

We start in \S 2 by setting some notation, as well as recalling a few well-known facts about currents we shall need. Next, in \S 3 we precisely define $C^{1,\alpha}$ tangentially immersed boundaries and boundaries having Lipschitz co-oriented mean curvature; these are Definitions \ref{immersedboundary},\ref{cmcboundary} respectively. We as well state for convenience the results from \cite{R18} which we shall need. In \S 4 we state and prove our main results, beginning with the asymptotic description near half-regular points Lemma \ref{halfregular} and concluding with the statement of our main result Theorem \ref{main}. We note that we omit the explicit proof of Theorem \ref{main}, as it is virtually identical to the proof of Theorem 5.3 of \cite{R18}. Finally, in the Appendix, for convenience we make some calculations in Lemma \ref{appendixlemma} needed in the proof of Lemma \ref{halfregular}, and state the general Hopf boundary point lemma needed here in Lemma \ref{appendixhopflemma}.

\subsection{Acknowledgements}

This work was partly conducted by the author while visiting the Korea Institute for Advanced Study, as an Associate Member.

\section{Notation}

We list basic notation and terminology we shall use throughout.

\begin{itemize}
 \item $\N,\R$ will denote the natural and real numbers respectively. We shall let $n \in \N$ with $n \geq 2.$ In this section we will let $\hat{n} \in \{ 1,\ldots,n \}.$
 \item We shall typically write points $x = (x_{1},\ldots,x_{n+1}) \in \R^{n+1}.$ Depending on context, we shall let 
 $$\R^{\hat{n}} = \{ (x_{1},\ldots,x_{\hat{n}},0,\ldots,0) \in \R^{n+1} : x_{1},\ldots,x_{\hat{n}} \in \R \}.$$
 We shall typically write points
 $$\begin{aligned}
 \xi = & (\xi,\ldots,\xi_{n-2}) = (\xi_{1},\ldots,\xi_{n-2},0,0,0) \in \R^{n-2} \text{ (if $n \geq 3$)}, \\
 z = & (z_{1},\ldots,z_{n-1}) = (z_{1},\ldots,z_{n-1},0,0) \in \R^{n-1}, \text{ and } \\
 y = & (y_{1},\ldots,y_{n}) =(y_{1},\ldots,y_{n},0) \in \R^{n}.
 \end{aligned}$$
 We will let $0$ denote the zero vector in different dimensions, depending on context. 
 \item If $n \geq 3,$ for each $\hat{n} \in \{ n-2,n-1,n \}$ let $\bop_{\hat{n}}:\R^{n+1} \rightarrow \R^{\hat{n}}$ denote the projection
 $$\bop_{\hat{n}}(x) = (x_{1},\ldots,x_{\hat{n}}).$$
 \item Let $e_{1},\ldots,e_{n+1}$ be the standard basis vectors for $\R^{n+1}.$
 \item For $A \subseteq \R^{n+1},$ let $\clos A$ denote the closure of $A.$ 
 \item We shall let $B_{\rho}(x)$ be the open ball in $\R^{n+1}$ of radius $\rho>0$ centered at $x.$ For $x \in \R^{\hat{n}},$ we write $B^{\hat{n}}_{\rho}(x) = B_{\rho}(x) \cap \R^{\hat{n}}.$
 \item For $x \in \R^{n+1}$ and $\lambda > 0,$ we let $\eta_{x,\lambda} : \R^{n+1} \rightarrow \R^{n+1}$ be the map $\eta_{x,\lambda}(\hat{x}) = \frac{\hat{x}-x}{\lambda}.$ We shall make use of $\eta_{-x,1},$ which is translation by $x.$ 
 \item We let $\ast : \bigwedge_{n} \R^{n+1} \rightarrow \R^{n+1}$ be the Hopf map
 $$\ast \left( \sum_{i=1}^{n+1} x_{i} (-1)^{i-1} e_{1} \wedge \ldots \wedge e_{i-1} \wedge e_{i+1} \wedge \ldots \wedge e_{n+1} \right) = (-1)^{n} \sum_{i=1}^{n+1} x_{i} e_{i}.$$
 Note that $\ast(e_{1} \wedge \ldots \wedge e_{n}) = e_{n+1}.$ 
 \item We shall let $D$ denote differentiation generally over $\R^{n+1}$ or $\R^{\hat{n}},$ depending on context. In the proof of Lemmas \ref{halfregulartangentcones},\ref{appendixlemma},\ref{appendixhopflemma} we will let $\D$ denote differentiation over $\R^{n-1}$ and $\uD$ differentiation over $\R^{n-2},$ for emphasis.
 
\bigskip

Also in the proof of Lemmas \ref{halfregulartangentcones},\ref{appendixlemma} we will consider functions $F=F(y,p)$ with $y \in \R^{n}$ and $p \in \R^{n-1}$ (in particular, for $G^{ij}$ and $H$ as in \eqref{hrtc17}). We denote with $k \in \{ 1,\ldots,n-1 \}$ the following derivatives: $\D_{k}F$ is the derivative of $F$ with respect to the $y_{k}$-variable; $D_{n}F$ is the derivative of $F$ with respect to the $y_{n}$-variable; $\frac{\partial F}{\partial p_{k}}$ is the derivative of $F$ with respect to the $p_{k}$-variable.
 \item $\HH^{\hat{n}}$ shall denote $\hat{n}$-dimensional Hausdorff measure in $\R^{n+1}.$ 
\end{itemize}
 
We now give notation related to currents in $\R^{n+1}.$ For a thorough introduction to currents, see \cite{F69},\cite{S83}.
 
\begin{itemize}
 \item Recall that $\mathcal{D}^{\hat{n}}(U)$ denotes for $U \subseteq \R^{n+1}$ an open set the smooth $\hat{n}$-forms compactly supported in $U.$
 \item For $T$ a current in $U \subseteq \R^{n+1}$ an open set and $f:U \rightarrow \R^{n+1},$ we denote $f_{\#} T$ the push-forward current of $T$ by $f$; we shall frequently make use of $\eta_{x,\lambda \#} T$ and in particular the translation $\eta_{-x,1 \#} T.$
 \item We say a current $\C$ is a cone if $\eta_{0,\lambda \#} \C = \C$ for every $\lambda>0.$
 \item Given an orientable $\hat{n}$-dimensional submanifold $M \subset \R^{n+1},$ we denote $\lsem M \rsem$ the associated multiplicity one current, given an orientation. 
 \item Denote by $\E^{\hat{n}}$ the $\hat{n}$-dimensional current in $\R^{n+1}$ given by $\E^{\hat{n}}(\omega) = \int_{\R^{\hat{n}}} \langle \omega, e_{1} \wedge \ldots \wedge e_{\hat{n}} \rangle \ d \HH^{\hat{n}}$ for $\omega \in \mathcal{D}^{\hat{n}}(\R^{n+1}).$
 \item For $U \subseteq \R^{n+1}$ an open set and $T$ an $\hat{n}$-dimensional current in $U,$ we let $\mu_{T}$ denote the associated mass measure of $T.$ This is given for $\hat{U}$ an open subset of $U$ by $\mu_{T}(\hat{U}) = \sup_{\omega \in \mathcal{D}^{\hat{n}}(\hat{U}),|\omega| \leq 1} T(\omega).$ As usual, we set $\spt T = \spt \mu_{T}.$
 
 \bigskip
 
 For $A$ a $\mu_{T}$-measurable set, we let $T \res A$ denote the restriction current $(T \res A)(\omega) = \int_{A} <\omega,\vec{T}> \ d \mu_{T}$ for $\omega \in \mathcal{D}^{\hat{n}}(U),$ where $\vec{T}$ is the orientation vector of $T.$ 

 \bigskip
 
Given $x \in U,$ we denote the density of $T$ at $x$ by
$$\Theta_{T}(x) = \lim_{\rho \searrow 0} \frac{\mu_{T}(B_{\rho}(x))}{\rho^{\hat{n}} \HH^{\hat{n}}(B^{\hat{n}}_{1}(0))},$$
whenever this limit exists.
 \item Given $U \subseteq \R^{n+1}$ an open set, we let $\I_{\hat{n},loc}(U)$ be the set of $\hat{n}$-dimensional currents $T$ so that $T,\partial T$ are respectively $\hat{n}$- and $(\hat{n}-1)$-rectifiable integer multiplicity. 
 
 \bigskip
 
 For $T \in \I_{\hat{n},loc},$ we let $T_{x}T$ denote the approximate tangent space of $T$ for the $\mu_{T}$-almost-every $x \in U$ such that this space exists; naturally, we let $T^{\perp}_{x}T$ denote the orthogonal complement of $T_{x}T$ in $\R^{n+1}.$
 \item For $T \in \I_{\hat{n},loc}(U),$ we denote $\delta T$ to be the first variation of mass, given by 
 $$\delta T (X) = \int \dive{T}{X} \ d \mu_{T}$$ 
 for $X \in C_{c}^{1}(U;\R^{n+1}).$ 
 
 \bigskip
 
We say that $T$ has mean curvature $H_{T}:U \rightarrow \R^{n+1}$ if $H_{T}$ is $\mu_{T}$-measurable and if
$$\delta T(X) = \int X \cdot H_{T} \ d \mu_{T}$$
 for every $X \in C_{c}^{1}(U;\R^{n+1})$
 \item For $T \in \I_{\hat{n},loc}(U),$ we let $\reg T$ denote the regular set of $T$: the set of $x \in \spt T$ so that there is a $\rho > 0$ such that $T \res B_{\rho}(x) = \theta \lsem M \rsem$ for $\theta \in \N$ and $M$ an $\hat{n}$-dimensional orientable (embedded) $C^{1}$ submanifold of $B_{\rho}(x).$ We define the singular set $\sing T = \spt T  \setminus \reg T.$
  \item We say $T \in \I_{\hat{n},loc}(U)$ is area-minimizing if $\mu_{T}(\hat{U}) \leq \mu_{R}(\hat{U})$ whenever $\hat{U} \subset U$ is an open set with $\clos \hat{U} \subset U$ and $R \in \I_{\hat{n},loc}(U)$ with $\partial R = \partial T$ and $\spt (T-R) \subset \hat{U}.$
 \item For $T \in \I_{\hat{n},loc}(U)$ area-minimizing there is, by Lemma 3.1 of \cite{B77} (see as well (2.10) of \cite{E89}), a $\mu_{\partial T}$-measurable vectorfield $\nu_{T}:U \rightarrow \R^{n+1}$ satisfying $|\nu_{T}| \leq 1$ for $\mu_{\partial T}$-almost-everywhere so that
 $$\delta T (X) = \int \nu_{T} \cdot X \ d \mu_{\partial T}$$
 for every $X \in C^{1}_{c}(U;\R^{n+1}).$ We call $\nu_{T}$ the generalized outward pointing normal of $\partial T$ with respect to $T.$ Note that since $|\delta T (X)| \leq \int |X \wedge \vec{\partial T}| \ d \mu_{\partial T}$ by Lemma 3.1 of \cite{B77} (see also (2.9) of \cite{E89}) for $X \in C^{1}_{c}(U;\R^{n+1}),$ we conclude $\nu_{T}(x) \in T^{\perp}_{x} \partial T$ for $\mu_{\partial T}$-almost-every $x \in U.$ 
\end{itemize}

\section{Definitions and previous results}

We now state for convenience the necessary results and definitions from \cite{R18}. We begin by defining precisely what it means to have $C^{1,\alpha}$ tangentially immersed boundary.

\begin{definition} \label{immersedboundary} Let $U \subseteq \R^{n+1}$ be an open subset and $\alpha \in (0,1].$ We define $\TI^{1,\alpha}_{n,loc}(U)$ to be the set of area-minimizing $T \in \I_{n,loc}(U)$ so that $\partial T$ is \emph{locally $C^{1,\alpha}$ tangentially immersed}: for every $x \in \spt \partial T$ there is $\rho>0,$ an orthogonal rotation $Q,$ and $N \in \N$ so that
$$\partial T \res B_{\rho}(x) = (-1)^{n} \sum_{\ell=1}^{N} m_{\ell} \Big[ (\eta_{-x,1} \circ Q \circ \Phi_{T,\ell})_{\#}(\E^{n-1} \res B^{n-1}_{\rho}(0)) \Big] \res B_{\rho}(x),$$
where for each $\ell =1,\ldots,N$ we have $m_{\ell} \in \N,$ and $\Phi_{T,\ell} \in C^{1,\alpha}(B^{n-1}_{\rho}(0);\R^{n+1})$ is the map
$$\Phi_{T,\ell}(z) = (z,\varphi_{T,\ell}(z),\psi_{T,\ell}(z)),$$
where $\varphi_{T,\ell},\psi_{T,\ell} \in C^{1,\alpha}(B^{n-1}_{\rho}(0))$ satisfy
$$\varphi_{T,\ell}(0)=\psi_{T,\ell}(0)=0 \text{ and } D \varphi_{T,\ell}(0) = D \psi_{T,\ell}(0)=0.$$ \end{definition}

This is Definition 3.1 of \cite{R18} in case $k=1.$ Observe that we could define what it means for a current to have $C^{1,\alpha}$ tangentially immersed boundary in general. But we include the requirement that $T \in \TI^{1,\alpha}_{n,loc}(U)$ must be area-minimizing for future brevity. Observe that if $T \in \TI^{1,\alpha}_{n,loc}(U)$ then the approximate tangent space $T_{x}\partial T$ exists for every $x \in \spt \partial T.$

\bigskip

In order to clearly state our results, we will state Definition \ref{conesdefinition}. First, to better understand Definition \ref{conesdefinition} and for convenience, we state Lemma 3.2 of \cite{R18}.

\begin{lemma} \label{cones} Suppose $\C \in \I_{n,loc}(\R^{n+1})$ is an area-minimizing cone with $\partial \C = m (-1)^{n} Q_{\#} \E^{n-1}$ for some $m \in \N$ and an orthogonal rotation $Q.$ Then $\C$ is of one of the following two forms:
\begin{enumerate}
  \item[(1)] There is $N \in \{ 1,\ldots,m \}$ and distinct orthogonal rotations $Q_{1},\ldots,Q_{N}$ about $\R^{n-1}$ so that
  $$\C = \sum_{k=1}^{N} m_{k} (Q \circ Q_{k})_{\#} (\E^{n} \res \{ y \in \R^{n}: y_{n}>0 \}),$$ 
  where $m_{1},\ldots,m_{N} \in \N$ satisfy $\sum_{k=1}^{N} m_{k} = m.$
  \item[(2)] There is $\theta \in \N$ so that 
  $$\C = Q_{\#} \Big( (m+\theta) \E^{n} \res \{ y \in \R^{n}: y_{n}>0 \} + \theta \E^{n} \res \{ y \in \R^{n}: y_{n}<0 \} \Big).$$
\end{enumerate}
\end{lemma}

\begin{proof} 
See the proof of Lemma 3.2 of \cite{R18}. 
\end{proof}

\bigskip

We thus give the following definition:

\begin{definition} \label{conesdefinition} Suppose $\C \in \I_{n,loc}(\R^{n+1}).$ If $\C$ is as in (1) of Lemma \ref{cones}, then we say that $\C$ is a \emph{sum of half-hyperplanes with constant orientation after rotation}. If $\C$ is as in (2) of Lemma \ref{cones}, then we say that $\C$ is a \emph{hyperplane with constant orientation but non-constant multiplicity}.  \end{definition}

We now state the main result of \cite{R18}, which is needed to prove Lemma \ref{halfregular}. We need Theorem \ref{hyperplanetangentconegraph} so that we can apply the general Hopf boundary point lemma of \cite{R17} to prove Lemma \ref{halfregular}.

\begin{theorem} \label{hyperplanetangentconegraph} Let $U \subseteq \R^{n+1}$ be an open set, $\alpha \in (0,1],$ and $T \in \TI^{1,\alpha}_{n,loc}(U).$ Suppose $x \in \spt \partial T$ and that $T$ at $x$ has a tangent cone which is a hyperplane with constant orientation but non-constant multiplicity (as in Definition \ref{conesdefinition}). Then there is a $\rho \in \dist (0,\dist(x,\partial U))$ and a solution to the minimal surface equation $u \in C^{\infty}(B^{n}_{\rho}(0))$ with $u(0)=0$ and $Du(0)=0$ such that
$$\spt T \cap B_{\rho}(x) = \eta_{-x,1}(Q(\gph{B^{n}_{\rho}(0)}{u})) \cap B_{\rho}(x)$$
for an orthogonal rotation $Q.$ The orientation vector for $T$ is given by
$$\ast \vec{T}(\tilde{x}) = Q \left( \left. \left( \frac{-Du}{\sqrt{1+|Du|^{2}}},\frac{1}{\sqrt{1+|Du|^{2}}} \right) \right|_{\bop_{n}(Q^{-1}(\tilde{x}-x))} \right)$$
if $\tilde{x} \in \spt T \cap B_{\rho}(x).$ \end{theorem}

\begin{proof}
This is Theorem 3.18 of \cite{R18}. Note that we should have, as stated above, $\bop_{n}(Q^{-1}(\tilde{x}-x))$ and not $\proj{\R^{n}}{\tilde{x}}$ as stated in \cite{R18}. 
\end{proof}

\bigskip

Next, we state a lemma which consequentially defines (and shows the existence of) \emph{half-regular} points at the boundary of $T \in \TI^{1,\alpha}_{n,loc}$ near any singular point of $\partial T.$ As previously noted, half-regular points were previously used to prove a main result of \cite{R18}, which we wish to extend here in Theorem \ref{main}. Lemma \ref{halfregular} is a simpler, but sufficient, version of Lemma 3.8 of \cite{R18}.

\begin{lemma} \label{halfregular} Let $U$ be an open subset of $\R^{n+1},$ $\alpha \in (0,1],$ and $T \in \TI^{1,\alpha}_{n,loc}(U).$ Suppose $x \in \sing \partial T$ and that $\rho \in (0,\dist(x,\partial U))$ is as in Definition \ref{immersedboundary}, so that
$$\partial T \res B_{\rho}(x) = (-1)^{n} \sum_{\ell=1}^{N} m_{\ell} \Big[ (\eta_{-x,1} \circ Q \circ \Phi_{T,\ell})_{\#}(\E^{n-1} \res B^{n-1}_{\rho}(0)) \Big] \res B_{\rho}(x)$$
for $N,m_{1},\ldots,m_{N} \in \N,$ an orthogonal rotation $Q,$ and $\Phi_{T,\ell} \in C^{1,\alpha}(B^{n-1}_{\rho}(0);\R^{n+1})$ for each $\ell = 1,\ldots,N.$ Then there is $\z \in B^{n-1}_{\rho}(0),$ a radius $\sigma \in (0,\rho-|\z|],$ and distinct $\ell,\tilde{\ell} \in \{ 1,\ldots,N \}$ so that
\begin{itemize}
 \item $(\eta_{-x,1} \circ Q \circ \Phi_{T,\ell})(B^{n-1}_{\sigma}(\z)) \cup (\eta_{-x,1} \circ Q \circ \Phi_{T,\tilde{\ell}})(B^{n-1}_{\sigma}(\z)) \subset \reg \partial T,$
 \item $\Phi_{T,\ell}(B^{n-1}_{\sigma}(\z)) \cap \Phi_{T,\tilde{\ell}}(B^{n-1}_{\sigma}(\z)) = \emptyset,$
 \item $\Phi_{T,\ell}(\partial B^{n-1}_{\sigma}(\z)) \cap \Phi_{T,\tilde{\ell}}(\partial B^{n-1}_{\sigma}(\z)) \neq \emptyset.$
\end{itemize}
With this, we say any point
$$\x \in \Phi_{T,\ell}(\partial B^{n-1}_{\sigma}(\z)) \cap \Phi_{T,\tilde{\ell}}(\partial B^{n-1}_{\sigma}(\z))$$
is \emph{half-regular}.
\end{lemma}

\begin{proof}
Follows directly from the statement of Lemma 3.8 of \cite{R18}.
\end{proof}

\bigskip

Finally, we precisely define what it means for an area-minimizing current $T$ to have boundary with Lipschitz co-oriented mean curvature. This and Definition \ref{immersedboundary} are our main concepts.

\begin{definition} \label{cmcboundary} 
Let $U$ be an open subset of $\R^{n+1},$ and suppose $T \in \I_{n,loc}(U)$ is area-minimizing. We say $\partial T$ has \emph{Lipschitz co-oriented mean curvature} if $\partial T$ has mean curvature $H_{\partial T} = h \nu_{T}$ for $h:U \rightarrow \R$ a Lipschitz function, where $\nu_{T}:U \rightarrow \R^{n+1}$ is the generalized outward pointing normal of $\partial T$ with respect to $T$; this means that
$$\int \dive{\partial T}{X} \ d \mu_{\partial T} = \int X \cdot (h \nu_{T}) \ d \mu_{\partial T}$$
for all $X \in C^{1}_{c}(U;\R^{n+1}).$ 
\end{definition}

The assumption that $T$ is area-minimizing in Definition \ref{cmcboundary} is merely to guarantee the existence of the generalized outward pointing unit normal $\nu_{T}$ of $\partial T$ with respect to $T$; see Lemma 3.1 of \cite{B77} and (2.10) of \cite{E89}. Definition \ref{cmcboundary} is a more specific version of Definition 4.1 of \cite{R18}, which does not require $h$ to be Lipschitz.

\section{Main results and proofs}

We now put our two main concepts together, and study co-dimension one area-minimizing currents $T$ with boundary being both $C^{1,\alpha}$ tangentially immersed and having co-oriented Lipschitz mean curvature. Again, our main result is Theorem \ref{main}, which states that near any $x \in \sing \partial T$ of such a current $T,$ either $T$ near $x$ exhibits a reasonable amount of regularity (as in Figure \ref{mainexample}) or $\spt T$ near $x$ must be extremely irregular. To prove Theorem \ref{main}, we must prove Theorem \ref{hyperplanetangentconeregularity}, which states that the boundary $\partial T$ of any such current $T$ is regular near any point $x \in \spt \partial T$ such that $T$ at $x$ has tangent cone which is a hyperplane with constant orientation but non-constant multiplicity. Given Theorem \ref{hyperplanetangentconeregularity}, then the proof of Theorem \ref{main} is virtually identical to the proof of Theorem 5.3 of \cite{R18}.

\bigskip

To prove Theorem \ref{hyperplanetangentconeregularity} we must prove Lemma \ref{halfregulartangentcones}, which shows that if $x \in \spt \partial T$ is half-regular (see Lemma \ref{halfregular}), then every tangent cone of $T$ at $x$ must be a sum of half-hyperplanes with constant orientation after rotation. The proof of Lemma \ref{halfregulartangentcones} will take the majority of this section. 

\bigskip

\begin{lemma} \label{halfregulartangentcones} Let $U \subseteq \R^{n+1}$ be an open set, and suppose $T \in \TI^{1,\alpha}_{n,loc}(U)$ where $\partial T$ has Lipschitz co-oriented mean curvature $H_{\partial T} = h \nu_{T}.$ For any $x \in \sing \partial T,$ there is $\rho \in (0,\dist(x,\partial U))$ so that for any half-regular $\x \in \sing \partial T \cap B_{\rho}(x)$ (see Lemma \ref{halfregular}), every tangent cone of $T$ at $\x$ is the sum of half-hyperplanes with constant orientation after rotation (see Definition \ref{conesdefinition}).
\end{lemma} 

\begin{proof}
Suppose (after translation) $x=0 \in \sing \partial T.$ Observe that $\TI^{1,\alpha}_{n,loc}(U) \subseteq \TI^{1,\beta}_{n,loc}(U)$ for each $\beta \in (0,\alpha].$ Replacing $\alpha$ with any $\hat{\alpha} \in (0,\min \{ \alpha, \frac{n-2}{n-1} \}),$ we can those suppose by Definition \ref{immersedboundary}, and after a rotation if necessary, that there is a $\rho \in (0,\dist(0,\partial U))$ such that there are $N,m_{1},\ldots,m_{N} \in \N$ and maps $\Phi_{T,\ell}(z) = (z,\varphi_{T,\ell}(z),\psi_{T,\ell}(z))$ for $z \in B^{n-1}_{\rho}(0)$ and each $\ell = 1,\ldots,N$ so that 
\begin{equation} \label{hrtc2} 
\begin{aligned}
& \partial T \res B_{\rho}(0) = (-1)^{n} \sum_{\ell=1}^{N} m_{\ell} \Phi_{T,\ell \#} (\E^{n-1} \res B^{n-1}_{\rho}(0)) \res B_{\rho}(0) \\
& \text{with } \varphi_{T,\ell},\psi_{T,\ell} \in C^{1,\alpha}(B^{n-1}_{\rho}(0)) \text{ for } \alpha \in \left(0,\frac{n-2}{n-1} \right), \text{ satisfying} \\
& \varphi_{T,\ell}(0)=\psi_{T,\ell}(0)=0 \text{ and } \D \varphi_{T,\ell}(0) = \D \psi_{T,\ell}(0)=0;
\end{aligned}
\end{equation}
recall that we shall let $\D$ denote differentiation over $\R^{n-1},$ for emphasis (see the ninth item in \S 2). We can also choose $\rho \in (0,\dist(x,\partial U))$ sufficiently small depending on $\epsilon=\epsilon(n)>0,$ to be determined later, so that 
\begin{equation} \label{hrtc3}
\| \D \varphi_{T,\ell} \|_{C(B^{n-1}_{\rho}(0))},\| \D \psi_{T,\ell} \|_{C(B^{n-1}_{\rho}(0))} < \epsilon. \end{equation}  

\bigskip

Suppose $\x \in \spt \partial T \cap B_{\rho/4}(0)$ is a half-regular point. Thus, by Lemma \ref{halfregular} there is $\z \in B^{n-1}_{\rho/4}(0)$ and $\sigma \in (0,\frac{\rho}{12})$ such that (after relabeling)
\begin{equation} \label{hrtc4}
\begin{aligned}
|\z - \bop_{n-1}(\x)| & = 3\sigma, \\
\x \in \Phi_{T,1}(\partial B^{n-1}_{3\sigma}(\z)) & \cap \Phi_{T,2}(\partial B^{n-1}_{3\sigma}(\z)), \\
\Phi_{T,1}(B^{n-1}_{3\sigma}(\z)) & \cup \Phi_{T,2}(B^{n-1}_{3\sigma}(\z)) \subset \reg \partial T, \text{ and} \\ 
\Phi_{T,1}(B^{n-1}_{3\sigma}(\z)) & \cap \Phi_{T,2}(B^{n-1}_{3\sigma}(\z)) = \emptyset;
\end{aligned}
\end{equation}
note that in applying Lemma \ref{halfregular} we have replaced $\sigma$ with $3\sigma.$

\bigskip

Consider any tangent cone $\C$ of $T$ at $\x.$ Then $\C$ is area-minimizing and by \eqref{hrtc2} there is an orthogonal rotation $Q$ so that
$$\partial \C = \left( \sum_{ \{ \ell \in \{ 1,\ldots,N \}: \x \in \Phi_{T,\ell}(B^{n-1}_{\rho}(0)) \}} m_{\ell} \right)(-1)^{n} Q_{\#} \E^{n-1}.$$
We conclude by Lemma \ref{cones} and Definition \ref{conesdefinition} that $\C$ is either a sum of half-hyperplanes with constant orientation after rotation or a hyperplane with constant orientation but non-constant multiplicity. Suppose for contradiction that $\C$ is a hyperplane with constant orientation but non-constant multiplicity. After applying Theorem \ref{hyperplanetangentconegraph} to $T$ at $\x,$ our goal is to apply Lemma \ref{appendixhopflemma} (a Hopf boundary point lemma) to $\partial T$ at $\x,$ with \eqref{hrtc4} in mind, in order to yield a contradiction. To properly do so will require that we carefully set some notation.

\bigskip

Proceeding, and to be clear, we assume for contradiction by \eqref{hrtc2},\eqref{hrtc3} that there is $m,\theta \in \N$ so that
\begin{equation} \label{hrtc5}
\begin{aligned}
& \begin{aligned}
\C = (Q \circ R)_{\#} \big( (m+\theta) & \E^{n} \res \{ y \in \R^{n} : y_{n}>0 \} \\
+ \theta & \E^{n} \res \{ y \in \R^{n} : y_{n}<0 \} \big)
\end{aligned} \\
& \text{for an orthogonal rotation } R \text{ about } \R^{n-1} \\
& \text{and an orthogonal rotation } Q \text{ with } \| Q-I_{n+1} \| < \CC \epsilon
\end{aligned}
\end{equation}
where $I_{n+1}$ is the $(n+1) \times (n+1)$ identity matrix and $\CC=\CC(n)>0;$ note that we mean $R$ fixes $\R^{n-1}.$ Also, $Q$ is such that $Q(\R^{n-1}) = T_{\x} \partial T.$ Observe as well that if $\epsilon = \epsilon(n)>0$ is sufficiently small then $|\bop_{n-1}(R^{-1}(Q^{-1}(z)))| \geq 2$ for all $z \in \partial B^{n-1}_{3}(0).$ We can thus take $S$ an orthogonal rotation so that by \eqref{hrtc4}
\begin{equation} \label{hrtc6}
\begin{aligned}
S & \left( \frac{\bop_{n-1} \left( R^{-1} \left( Q^{-1} \left( \eta_{\bop_{n-1}(\x),\sigma}(\z) \right) \right) \right)}{\left| \bop_{n-1} \left( R^{-1} \left( Q^{-1} \left( \eta_{\bop_{n-1}(\x),\sigma}(\z) \right) \right) \right) \right|} \right) =  e_{n-1}, \\
S & (e_{n}) = e_{n}, \ S(e_{n+1})=e_{n+1}.
\end{aligned}
\end{equation}
We shall use $S$ so that we can define a region $\hat{\Omega}$ over which we can apply Lemma \ref{appendixhopflemma} as stated.

\bigskip

We now translate and rotate $T$ so that we can apply Lemma \ref{appendixhopflemma} at the origin. Define the orthogonal rotation and the current respectively 
\begin{equation} \label{hrtc7}
\hat{Q} = S \circ R^{-1} \circ Q^{-1} \text{ and } \hat{T} = (\hat{Q} \circ \eta_{\x,\sigma})_{\#} T.
\end{equation}
We now choose $\epsilon=\epsilon(n)>0,$ as well as $\z \in B^{n-1}_{\rho/4}(0)$ and $\sigma \in (0,\frac{\rho}{12})$ as described if necessary, so that the following three occur.

\bigskip

First, choose $\epsilon=\epsilon(n)>0$ sufficiently small so that by \eqref{hrtc3},\eqref{hrtc4},\eqref{hrtc6},\eqref{hrtc7} (in particular the first identity of \eqref{hrtc6}) we have for each $\ell=1,2$
\begin{equation} \label{hrtc8} 
\begin{aligned}
& \begin{aligned}
(B^{n-2}_{1} & (0) \times (-3,3)) \cap \left( \bop_{n-1} \circ \hat{Q} \circ \eta_{\x,\sigma} \circ \Phi_{T,\ell} \right) \left( B^{n-1}_{3\sigma}(\z) \right) \\
& = \{ z \in B^{n-2}_{1}(0) \times (-3,3): z_{n-1} > w_{\ell}(\bop_{n-2}(z)) \} \subset B^{n-1}_{7}(0)
\end{aligned} \\
& \text{where } w_{\ell} \in C^{1,\alpha}(B^{n-2}_{1}(0)) \text{ satisfies} \\
& w_{\ell}(0)=0, \ \uD w_{\ell}(0)=0, \text{ and } \| w_{\ell} \|_{C^{1,\alpha}(B^{n-2}_{1}(0))} < \epsilon;
\end{aligned}
\end{equation}
recall that we shall let $\uD$ denote differentiation over $\R^{n-2},$ for emphasis (see the ninth item in \S 2). The functions $w_{\ell}$ shall be used to define the region $\hat{\Omega}$ over which we shall apply Lemma \ref{appendixhopflemma}. But the next step will be necessary to define $\hat{\Omega}.$

\bigskip

Second, taking $\sigma>0$ smaller and $\z$ closer to $\x$ in that direction if necessary (more specifically, replacing $\sigma$ with $\hat{\sigma} \in (0,\sigma)$ and $\z$ with $\hat{\z} = \x + \frac{\hat{\sigma}}{3} (\z- \bop_{n-1}(\x))$), we can assume $\z \in B^{n-1}_{\rho/4}(0)$ and $\sigma \in (0,\frac{\rho}{12})$ are such that \eqref{hrtc4},\eqref{hrtc8} still hold while by \eqref{hrtc2},\eqref{hrtc5},\eqref{hrtc6},\eqref{hrtc7} we have that Definition \ref{immersedboundary} holds for $\hat{T}$ at the origin with radius $=8:$ there are $\hat{N},\hat{m}_{1},\ldots,\hat{m}_{\hat{N}} \in \N$ and maps $\Phi_{\hat{T},\ell}(z) =(z,\varphi_{\hat{T},\ell}(z),\psi_{\hat{T},\ell}(z))$ for $z \in B^{n-1}_{8}(0)$ and each $\ell = 1,\ldots,\hat{N}$ so that
\begin{equation} \label{hrtc9}
\begin{aligned}
& \partial \hat{T} \res B_{8}(0) = (-1)^{n} \sum_{\ell=1}^{\hat{N}} \hat{m}_{\ell} \Phi_{\hat{T},\ell \#} (\E^{n-1} \res B^{n-1}_{8}(0)) \res B_{8}(0) \\
& \text{with } \varphi_{\hat{T},\ell},\psi_{\hat{T},\ell} \in C^{1,\alpha}(B^{n-1}_{8}(0)) \text{ satisfying} \\
& \varphi_{\hat{T},\ell}(0) = \psi_{\hat{T},\ell}(0) = 0, \ \D \varphi_{\hat{T},\ell}(0) = \D \psi_{\hat{T},\ell}(0) = 0, \text{ and } \\
& \| \varphi_{\hat{T},\ell}\|_{C^{1,\alpha}(B^{n-1}_{8}(0))}, \| \psi_{\hat{T},\ell} \|_{C^{1,\alpha}(B^{n-1}_{8}(0))} < \epsilon;
\end{aligned}
\end{equation}
note that $8\sigma < \frac{3\rho}{4}$ and $\hat{N} \in \{ 1,\ldots,N \}.$ After relabeling, we define and conclude by \eqref{hrtc4},\eqref{hrtc7},\eqref{hrtc8},\eqref{hrtc9} and Definition \ref{immersedboundary} that
\begin{equation} \label{hrtc10}
\begin{aligned}
& \begin{aligned}
w(\xi) = & \max \{ w_{1}(\xi),w_{2}(\xi) \} \text{ for } \xi \in B^{n-2}_{1}(0), \\
\Omega = & \{ z \in B^{n-2}_{1}(0) \times (-3,3): z_{n-1} > w(\bop_{n-2}(z)) \}, \\
0 \in & \Phi_{\hat{T},1}(\partial \Omega) \cap \Phi_{\hat{T},2}(\partial \Omega), \\
& \Phi_{\hat{T},1}(\Omega) \cup \Phi_{\hat{T},2}(\Omega) \subset \reg \partial \hat{T}, \text{ and} \\ 
& \Phi_{\hat{T},1}(\Omega) \cap \Phi_{\hat{T},2}(\Omega) = \emptyset,
\end{aligned} \\
& \text{with } w \in C^{1,\alpha}(B^{n-2}_{1}(0)) \text{ satisfying } \\
& w(0)=0, \ \uD w(0)=0, \text{ and } \| w \|_{C^{1,\alpha}(B^{n-2}_{1}(0))} < \epsilon.
\end{aligned}
\end{equation}
We will later modify the function $w,$ for technical reasons, in order to define the region $\hat{\Omega}$ over which we will apply Lemma \ref{appendixhopflemma} to yield a contradiction.

\bigskip

Third, we can also assume $\z \in B^{n-1}_{\rho/4}(0)$ and $\sigma \in (0,\frac{\rho}{12})$ are such that \eqref{hrtc8},\eqref{hrtc9},\eqref{hrtc10} continue to hold while Theorem \ref{hyperplanetangentconegraph} holds for $\hat{T}$ at the origin with radius $=8.$ More specifically, by \eqref{hrtc5},\eqref{hrtc6},\eqref{hrtc7} and Theorem \ref{hyperplanetangentconegraph} there is $u \in C^{\infty}(B^{n}_{8}(0))$ so that
\begin{equation} \label{hrtc11}
\begin{aligned}
\spt \hat{T} \cap B_{8}(0) = & \gph{B^{n}_{8}(0)}{u} \cap B_{8}(0), \\
u(0) = & 0 \text{ and } Du(0) = 0 \text{ with } \| u \|_{C^{2}(B^{n}_{8}(0))} < \epsilon, \\
\ast \hat{T} (y,u(y)) = & \frac{(-Du(y),1)}{\sqrt{1+|Du(y)|^{2}}} \text{ for } (y,u(y)) \in B_{8}(0);
\end{aligned}
\end{equation}
note that we specifically used $S(e_{n+1})=e_{n+1}$ by \eqref{hrtc6}. The graph of $u,$ together with Definition \ref{cmcboundary}, will allow us to compute the appropriate uniformly elliptic weak equation satisfied by $\partial T$ (heuristically speaking), to which we will apply Lemma \ref{appendixhopflemma}.

\bigskip

Consider $\varphi_{\hat{T},\ell}$ for $\ell = 1,2.$ Our goal is to show that we can apply Lemma \ref{appendixhopflemma} to $s = \varphi_{\hat{T},2} - \varphi_{\hat{T},1}$ at the origin. For this, fix $\ell \in \{ 1,2 \}$ and for simplicity of notation (until otherwise stated) write $\varphi = \varphi_{\hat{T},\ell}.$ We must show $\varphi$ satisfies a divergence-form equation. We do so in the following four steps, which will require setting some notation.

\bigskip 

First note that by \eqref{hrtc9},\eqref{hrtc11} with $\epsilon=\epsilon(n)>0$ sufficiently small we have
\begin{equation}
\label{hrtc12}
\Phi_{\hat{T},\ell}(\Omega) = \{ (z,\varphi(z),u(z,\varphi(z))): z \in \Omega \}.
\end{equation}
To further simplify notation, define for $y \in B^{n}_{8}(0)$ and $z \in \Omega$
\begin{equation} \label{hrtc13}
\begin{aligned}
\hat{h}(y) = & \sigma h \Big((\hat{Q} \circ \eta_{\x,\sigma})^{-1} \big( y,u(y) \big) \Big), \\
u_{i}(y) = & D_{i}u(y) = e_{i} \cdot Du(y) \text{ for } i \in \{ 1,\ldots,n \}, \\
\varphi_{j}(z) = & \D_{j} \varphi(z) = e_{j} \cdot \D \varphi(z) \text{ and } \\
\partial_{j}(z) = & (e_{j},\varphi_{j}(z),u_{j}(z,\varphi(z))+u_{n}(z,\varphi(z)) \varphi_{j}(z)) \\
& \text{ for } j \in \{ 1,\ldots,n-1 \},
\end{aligned}
\end{equation}
where recall for emphasis we let $\D$ be the gradient over $\R^{n-1}.$ 

\bigskip

Second, consider the downward pointing unit normal of the graph of $\varphi$ over $\Omega$ within the graph of $u.$ This is given by $\nu(z,\varphi(z),\D \varphi(z)) \text{ for } z \in \Omega$ where we define $\nu: B^{n}_{8}(0) \times \R^{n-1} \rightarrow \R^{n+1}$ by
\begin{equation} \label{hrtc14}
\nu(y,p) = \frac{ \left( (p,-1,0) + \left( \frac{(p,-1) \cdot Du(y)}{1+|Du(y)|^{2}} \right) ( -Du(y),1 ) \right)}{\sqrt{1+|p|^{2} - \frac{((p,-1) \cdot Du(y))^{2}}{1+|Du(y)|^{2}}}}
\end{equation}
for $y \in B^{n}_{8}(0)$ and $p \in \R^{n-1}.$

\bigskip 

Third, we define the matrix giving the metric tensor of $\Phi_{T,\ell}(B^{n-1}_{8}(0)),$ and the entries of its inverse. For each $p \in \R^{n-1}$ and $q \in \R^{n}$ define
\begin{equation} \label{hrtc15}
\begin{aligned}
g(p,q) = I_{n-1} & +(1+q_{n}^{2}) pp^{T}+q_{n} \big( p \bop_{n-1}(q)^{T} + \bop_{n-1}(q) p^{T} \big) \\
& + \bop_{n-1}(q) \bop_{n-1}(q)^{T}, \text{ and } \\
g^{ij}(p,q) = e_{i} & \cdot g(p,q)^{-1} e_{j} \text{ for } i,j \in \{ 1,\ldots,n-1 \},
\end{aligned}
\end{equation}
where $I_{n-1}$ is the $(n-1) \times (n-1)$ identity matrix; $g$ is generally invertible, but we may by \eqref{hrtc9},\eqref{hrtc11} restrict $g$ to $|p|,|q|$ small. Consequentially, we can choose $\epsilon = \epsilon(n)>0$ sufficiently small so that
\begin{equation} \label{hrtc16}
\| g(p,q) - I_{n-1} \| < \epsilon \text{ and } \| g(p,q)^{-1}- I_{n-1} \| < \epsilon
\end{equation}
for each $i,j \in \{ 1,\ldots,n-1 \}$ whenever $p \in B^{n-1}_{\epsilon}(0)$ and $q \in B^{n}_{\epsilon}(0).$

\bigskip

Fourth, we apply Definition \ref{cmcboundary}. Take any $\zeta \in C^{1}_{c}(\Omega).$ Using Definition \ref{cmcboundary} with the vector field
$$X(x) = Z(x) \zeta(\bop_{n-1}(x)) e_{n} \text{ for } x \in \Omega \times \R^{2},$$
where $Z: \R^{n+1} \rightarrow \R$ is a function of the form $Z(x) = Z(x_{n},x_{n+1})$ for $x \in \R^{n+1},$ we can show using as well \eqref{hrtc7},\eqref{hrtc9}-\eqref{hrtc15} that 
$$\begin{aligned}
\int & \sum_{i,j=1}^{n-1} \left\{ \begin{aligned}
& g^{ij}(\D \varphi(z),Du(z,\varphi(z))) \\
& \times D_{\partial_{i}(z)}( \zeta(z) e_{n}) \cdot \partial_{j}(z) \\
& \times \sqrt{\det \Big( g \big( \D \varphi(z),Du(z,\varphi(z)) \big) \Big)}  
\end{aligned} \right\} d \HH^{n-1}(z) \\
& = \int \left\{ \begin{aligned}
& \hat{h}(z,\varphi(z),u(z,\varphi(z))) \\
& \times \nu(z,\varphi(z),u(z,\varphi(z))) \cdot \zeta(z) e_{n} \\
& \times \sqrt{\det \Big( g \big( \D \varphi(z),Du(z,\varphi(z)) \big) \Big)} 
\end{aligned} \right\} d \HH^{n-1}(z).
\end{aligned}$$
Computing using \eqref{hrtc13},\eqref{hrtc14}, we conclude $\varphi$ is a weak solution over $\Omega$ to the divergence-form equation
\begin{equation} \label{hrtc17}
\begin{aligned}
\sum_{i,j=1}^{n-1} & \D_{i} \left( G^{ij}(z,\varphi(z),\D \varphi(z)) \D_{j} \varphi(z) \right) = H(z,\varphi(z),\D \varphi(z)) \text{ with} \\
& \begin{aligned}
G^{ij}(y,p) = & \sqrt{\det(g(p,Du(y)))} g^{ij}(p,Du(y)) \text{ and} \\
H(y,p) = & \left\{ \begin{aligned}
& \hat{h}(y,u(y)) \left( \frac{ 1 + \left( \frac{ (p,-1) \cdot Du(y)}{1+|Du(y)|^{2}} \right) u_{n}(y)}{\sqrt{1+|p|^{2} - \frac{((p,-1) \cdot Du(y))^{2}}{1+|Du(y)|^{2}}}} \right) \\
& \times \sqrt{\det(g(p,Du(y)))}
\end{aligned} \right.
\end{aligned} \\
& \text{for } y \in B^{n}_{8}(0) \text{ and } p \in \R^{n-1}, \text{ where} \\
& G^{ij} \in C^{\infty}(B^{n}_{8}(0) \times \R^{n-1}) \text{ for each } i,j \in \{ 1,\ldots,n-1 \} \text{ and} \\
& H:B^{n}_{8}(0) \times \R^{n-1} \rightarrow \R \text{ is Lipschitz.} 
\end{aligned}
\end{equation}
Choosing $\epsilon = \epsilon(n)>0$ sufficiently small, we conclude by \eqref{hrtc11},\eqref{hrtc16} that
\begin{equation} \label{hrtc18}
\sum_{i,j=1}^{n-1} G^{ij}(y,p) \hat{p}_{i} \hat{p}_{j} \geq \frac{3}{4} |\hat{p}|^{2} \text{ for all } \hat{p} \in \R^{n-1}
\end{equation}
whenever $y \in B^{n}_{8}(0)$ and $p \in B^{n-1}_{\epsilon}(0).$ 

\bigskip

Before we set and consider $s = \varphi_{\hat{T},2} - \varphi_{\hat{T},1}$ (to which we shall apply Lemma \ref{appendixhopflemma}), we must use \eqref{hrtc17} to derive second derivative estimates for $\varphi.$ Since $\varphi \in C^{1,\alpha}(\overline{\Omega})$ by \eqref{hrtc9}, then standard Schauder theory together with \eqref{hrtc17},\eqref{hrtc18} imply $\varphi \in C^{2,\alpha}(\Omega).$ One can show this more directly using the gradient estimates from \cite{GT83} and difference quotient methods. The calculations which follow illustrate the argument. 

\bigskip

Define and conclude by \eqref{hrtc10} with $\epsilon=\epsilon(n)>0$ sufficiently small
\begin{equation} \label{hrtc19}
\begin{aligned}
\hat{w}(\xi) = & 2w(\xi)+ 4 \epsilon |\xi|^{1+\alpha} \text{ for } \xi \in B^{n-2}_{1}(0) \text{ where } \\
\hat{w} \in & C^{1,\alpha}(B^{n-2}_{1}(0)) \text{ satisfies} \\
\hat{w}(0) = & 0, \ \uD \hat{w}(0) = 0 \text{ and } \| \uD \hat{w} \|_{C(B^{n-2}_{1}(0))} < 6 \epsilon, \\
\hat{\Omega} = & \{ z \in B^{n-2}_{1}(0) \times (-3,3): z_{n-1} > \hat{w}(\bop_{n-2}(z)) \}, \\
\hat{\Omega} \subset & \Omega \cap \{ z \in B^{n-2}_{1}(0) \times (0,3): z_{n-1}>2 \epsilon |\bop_{n-2}(z)|^{1+\alpha} \}, \\
\text{and } & \clos B^{n-1}_{\frac{\hat{z}_{n-1}}{4}}(\hat{z}) \subset \Omega \text{ when } \hat{z} \in \hat{\Omega}.
\end{aligned}
\end{equation}
Let us show the last claim. Fix $\hat{z} \in \hat{\Omega},$ then the fifth item of \eqref{hrtc19} implies $\hat{z}_{n-1} > 0.$ Moreover, for any $z \in \clos B^{n-1}_{\frac{\hat{z}_{n-1}}{4}}(\hat{z})$ we have by \eqref{hrtc10} and the definition of $\hat{w},\hat{\Omega}$ with $\epsilon=\epsilon(n) \in (0,1)$
$$\begin{aligned}
z_{n-1} = & w(\bop_{n-2}(z)) + \hat{z}_{n-1} + z_{n-1} - \hat{z}_{n-1} \\
& + w(\bop_{n-2}(\hat{z})) - w(\bop_{n-2}(z)) - w(\bop_{n-2}(\hat{z})) \\
\geq & w(\bop_{n-2}(z)) + \hat{z}_{n-1} - |z_{n-1}-\hat{z}_{n-1}|  \\
& - \epsilon | \bop_{n-2}(\hat{z}) - \bop_{n-2}(z) | - w(\bop_{n-2}(\hat{z})) \\
\geq & w(\bop_{n-2}(z)) + \frac{\hat{z}_{n-1}}{2} - w(\bop_{n-2}(\hat{z})) \\
> & w(\bop_{n-2}(z)) + 2 \epsilon |\bop_{n-2}(z)|^{1+\alpha} \geq w(\bop_{n-2}(z)).
\end{aligned}$$
We now note that $\hat{\Omega}$ is precisely the region over which we shall apply Lemma \ref{appendixhopflemma}. But now we bound the second derivative of $\varphi$ over $\hat{\Omega}.$

\bigskip

Since $\clos B^{n-1}_{\frac{\hat{z}_{n-1}}{4}}(\hat{z}) \subset \Omega,$ we can consider $\zeta \in C^{1}_{c}(B^{n-1}_{\frac{\hat{z}_{n-1}}{4}}(\hat{z}))$ and integrate \eqref{hrtc17} against $\D_{k} \zeta$ with $k \in \{ 1,\ldots,n-1 \}$ fixed. Recalling that $\varphi \in C^{2,\alpha}(\Omega)$ by Schauder theory, then using integration by parts we can define and conclude with $k \in \{ 1,\ldots,n-1 \}$ fixed that
$$\varphi_{k} = \D_{k} \varphi \in C^{1,\alpha}(\clos B^{n-1}_{\frac{\hat{z}_{n-1}}{4}}(\hat{z}))$$
is a weak solution over $B^{n-1}_{\frac{\hat{z}_{n-1}}{4}}(\hat{z})$ to the equation
\begin{equation} \label{hrtc20}
\sum_{i,j=1}^{n-1} \D_{i} \left( A^{ij} \D_{j} \varphi_{k} \right) + \sum_{i=1}^{n-1} \D_{i} \left( B^{i} \varphi_{k} \right) = \sum_{i=1}^{n-1} \D_{i} F^{i}
\end{equation}
where by \eqref{hrtc9},\eqref{hrtc17} for $z \in \clos B^{n-1}_{\frac{\hat{z}_{n-1}}{4}}(\hat{z})$ and $i,j \in \{ 1,\ldots,n-1 \}$ we define
$$\begin{aligned}
A^{ij}(z) = & G^{ij}(z,\varphi(z),\D \varphi(z)) +  \sum_{\hat{j}=1}^{n-1} \left( \frac{\partial G^{i \hat{j}}}{\partial p_{j}} \right)(z,\varphi(z),\D \varphi(z)) \D_{\hat{j}} \varphi(z) \\
& A^{ij} \in C^{0,\alpha}(\clos B^{n-1}_{\frac{\hat{z}_{n-1}}{4}}(\hat{z})), \\
B^{i}(z) = & \sum_{\hat{j}=1}^{n-1} (D_{n} G^{i \hat{j}}) (z,\varphi(z),\D \varphi(z)) \D_{\hat{j}} \varphi(z) \\
& B^{i} \in C^{0,\alpha}(\clos B^{n-1}_{\frac{\hat{z}_{n-1}}{4}}(\hat{z})), \\
F^{i}(z) = & \delta_{ik} H(z,\varphi(z),\D \varphi(z)) - \sum_{\hat{j}=1}^{n-1} (\D_{k} G^{i \hat{j}})(z,\varphi(z),\D \varphi(z)) \D_{\hat{j}} \varphi(z) \\
& F^{i} \in C^{0,\alpha}(\clos B^{n-1}_{\frac{\hat{z}_{n-1}}{4}}(\hat{z}));
\end{aligned}$$
observe that we use the notation for derivatives discussed in the ninth item of \S 2. We now use \eqref{hrtc20} together with standard Schauder estimates from \cite{GT83} to bound $|\D \varphi_{k}(\hat{z})|.$

\bigskip 

We must thus observe the equation \eqref{hrtc20} is uniformly elliptic, for which we use the calculations of Lemma \ref{appendixlemma}. To see this, note that \eqref{hrtc9},\eqref{hrtc19} with $\epsilon=\epsilon(n)>0$ sufficiently small imply $(z,\varphi(z)) \in B^{n}_{8}(0)$ and $\D \varphi(z)  \in B^{n-1}_{\epsilon}(0)$ for each $z \in B^{n-1}_{\frac{\hat{z}_{n-1}}{4}}(\hat{z}).$ We conclude by \eqref{hrtc9},\eqref{hrtc18},\eqref{app5} and the definition of $A^{ij}$ above that
$$\sum_{i,j=1}^{n-1} A^{ij}(z) p_{i}p_{j} \geq \frac{1}{2} |p|^{2} \text{ for all } p \in \R^{n-1},$$
when $z \in B^{n-1}_{\frac{\hat{z}_{n-1}}{4}}(\hat{z}),$ if $\epsilon=\epsilon(n)>0$ is sufficiently small. Recalling that
$$\varphi_{k} = \D_{k} \varphi_{\hat{T},\ell} \text{ for } \ell=1,2 \text{ and } k \in \{ 1,\ldots,n-1 \},$$
then the demonstrated uniform ellipticity of the $A^{ij}$ together with \eqref{hrtc9},\eqref{hrtc11},\eqref{hrtc15},\eqref{hrtc16},\eqref{hrtc17},\eqref{hrtc20} means that we can apply Theorem 8.32 of \cite{GT83} over $B^{n-1}_{\frac{\hat{z}_{n-1}}{4}}(\hat{z}).$ We conclude there is a $\CC>0$ depending on $n,$ $\sup_{B^{n}_{8}(0) \times B^{n}_{\epsilon}(0)} |H|,$ and the Lipschitz constant of $H$ so that for each $\ell=1,2$ and $k \in \{ 1,\ldots, n-1 \}$
\begin{equation} \label{hrtc21}
|\D \D_{k} \varphi_{\hat{T},\ell}(\hat{z})| \leq \frac{\CC}{\hat{z}_{n-1}} \text{ for } \hat{z} \in \hat{\Omega};
\end{equation}
note that we in particular used $\| \varphi_{\hat{T},\ell} \|_{C^{1,\alpha}(B^{n-1}_{8}(0))},\| u \|_{C^{2}(B^{n}_{8}(0))} < \epsilon.$

\bigskip

Now define and conclude by \eqref{hrtc9},\eqref{hrtc19}
\begin{equation} \label{hrtc22}
s =\varphi_{\hat{T},2} - \varphi_{\hat{T},1} \in C^{1,\alpha}(\clos \hat{\Omega}). 
\end{equation}
Observe as well by \eqref{hrtc10},\eqref{hrtc12},\eqref{hrtc19} that
$$\begin{aligned}
\emptyset = & \Phi_{\hat{T},1}(\hat{\Omega}) \cap \Phi_{\hat{T},2}(\hat{\Omega}) \\
= & \{ (z,\varphi_{\hat{T},1}(z),u(x,\varphi_{\hat{T},1}(z))) : z \in \hat{\Omega} \} \\
& \cap \{ (z,\varphi_{\hat{T},2}(z),u(x,\varphi_{\hat{T},2}(z))) : z \in \hat{\Omega} \}.
\end{aligned}$$
We conclude $\varphi_{\hat{T},1}(z) \neq \varphi_{\hat{T},2}(z)$ for each $z \in \hat{\Omega}.$ After relabeling, we can suppose together with \eqref{hrtc9} that
\begin{equation} \label{hrtc23}
s(z) > s(0)=0 \text{ for each } z \in \hat{\Omega} \text{ and } \D s(0) = 0.
\end{equation}
We wish to apply Lemma \ref{appendixhopflemma} (a Hopf boundary point lemma) to $s$ at the origin. To do so, we must verify that $s$ satisfies a uniformly elliptic equation with $C^{0,\alpha}$ top order coefficients, and appropriately bounded lower order coefficients. 

\bigskip

Using \eqref{hrtc17} and one-dimensional calculus, we conclude $s$ is a weak solution over $\hat{\Omega}$ to the equation 
\begin{equation} \label{hrtc24}
\sum_{i,j=1}^{n-1} \D_{i} \left( a^{ij} \D_{j}s \right) + \sum_{i=1}^{n-1} c^{i} \D_{i}s + ds = 0
\end{equation}
where we define for each $i,j \in \{ 1,\ldots,n-1 \}$
\begin{equation} \label{hrtc25}
\begin{aligned}
a^{ij}(z) = & \int_{0}^{1} \left\{ \begin{aligned}
& G^{ij}(Y(t,z),P(t,z)) \\
& + \sum_{\hat{j}=1}^{n-1} \left( \frac{\partial G^{i \hat{j}}}{\partial p_{j}} \right)(Y(t,z),P(t,z)) P_{\hat{j}}(t,z)
\end{aligned} \right\} \ dt, \\
c^{i}(z) = & - \int_{0}^{1} \left\{ \begin{aligned}
& \sum_{\hat{j}=1}^{n-1} (D_{n}G^{i \hat{j}})(Y(t,z),P(t,z)) P_{\hat{j}}(t,z) \\
& + \left( \frac{\partial H}{\partial p_{i}} \right)(Y(t,z),P(t,z)) 
\end{aligned} \right\} \ dt, \\
d(z) = & - \int_{0}^{1} \left\{ \begin{aligned}
& \sum_{\hat{i},\hat{j}=1}^{n-1} \D_{\hat{i}} \left( (D_{n}G^{\hat{i} \hat{j}})(Y(t,z),P(t,z)) P_{\hat{j}}(t,z) \right) \\
& + (D_{n}H)(Y(t,z),P(t,z))
\end{aligned} \right\} \ dt, \\
\text{with } & Y(t,z) = (z,t \varphi_{\hat{T},2}(z)+(1-t) \varphi_{\hat{T},1}(z)), \\
& P(t,z) = t \D \varphi_{\hat{T},2}(z)+(1-t) \D \varphi_{\hat{T},1}(z), \\
\text{and } & P_{\hat{j}}(t,z) = e_{\hat{j}} \cdot P(t,z) \text{ for each } \hat{j} \in \{ 1,\ldots,n-1 \},
\end{aligned}
\end{equation}
for $t \in [0,1]$ and $z \in \hat{\Omega}.$ Observe that by \eqref{hrtc9},\eqref{hrtc19} we have 
\begin{equation} \label{hrtc26}
Y(t,z) \in B^{n}_{5}(0) \text{ and } P(t,z) \in B^{n-1}_{\epsilon}(0)
\end{equation}
for $t \in [0,1]$ and $z \in \hat{\Omega}$ if $\epsilon=\epsilon(n)>0$ is sufficiently small. We conclude by \eqref{hrtc9},\eqref{hrtc18},\eqref{hrtc25},\eqref{app5} that
\begin{equation} \label{hrtc27}
\sum_{i,j=1}^{n-1} a^{ij}(z) p_{i}p_{j} \geq \frac{1}{2} |p|^{2} \text{ for all } p \in \R^{n-1}
\end{equation}
when $z \in \hat{\Omega},$ if $\epsilon=\epsilon(n)>0$ is sufficiently small. We as well have that $Y(t,0)=0$ and $P(t,0)=0$ for each $t \in [0,1]$ by \eqref{hrtc9}, and so by \eqref{hrtc11},\eqref{hrtc15},\eqref{hrtc17} we have
\begin{equation} \label{hrtc28}
a^{ij}(0) = \delta_{ij} = a^{ji}(0) \text{ for each } i,j \in \{ 1,\ldots,n-1 \}.
\end{equation}
In order to apply Lemma \ref{appendixhopflemma}, we must now consider the boundedness of the coefficients of the equation given in \eqref{hrtc25}.

\bigskip

First, we readily conclude by \eqref{hrtc9},\eqref{hrtc17} that
\begin{equation} \label{hrtc29}
a^{ij} \in C^{0,\alpha}(\clos \hat{\Omega}) \text{ and } c^{i} \in L^{\infty}(\hat{\Omega})
\end{equation}
for each $i,j \in \{ 1,\ldots,n-1 \}.$ We as well claim $d \in L^{q}(\hat{\Omega})$ for each $q \in \left( n-1,\frac{n-2}{\alpha} \right).$ This is not immediately evident, and so we carefully verify in what follows.

\bigskip

We can compute in \eqref{hrtc25} that
\begin{equation} \label{hrtc30}
\begin{array}{c}
d(z) = -\int_{0}^{1} \hspace{1in} dt \\
\overbrace{\hspace{4in}} \\
\begin{array}{lr}
\sum_{\hat{i},\hat{j}=1}^{n-1} (\D_{\hat{i}}D_{n} G^{\hat{i} \hat{j}})(Y(t,z),P(t,z)) P_{\hat{j}}(t,z) & (a) \\
+ \sum_{\hat{i},\hat{j}=1}^{n-1} (D^{2}_{n}G^{\hat{i} \hat{j}})(Y(t,z),P(t,z)) P_{\hat{i}}(t,z) P_{\hat{j}}(t,z) & (b) \\
+ \sum_{\hat{i},\hat{j}=1}^{n-1} \left\{ \begin{aligned}
& (D_{n}G^{\hat{i} \hat{j}})(Y(t,z),P(t,z)) \\
& \times \left( t \D_{\hat{i}} \D_{\hat{j}} \varphi_{\hat{T},2}(z) + (1-t) \D_{\hat{i}} \D_{\hat{j}} \varphi_{\hat{T},1}(z) \right) 
\end{aligned} \right. & (c) \\
+ \sum_{\hat{i},\hat{j},k=1}^{n-1} \left\{ \begin{aligned}
& \left( \frac{\partial D_{n}G^{\hat{i} \hat{j}}}{\partial p_{k}} \right)(Y(t,z),P(t,z)) P_{\hat{j}}(t,z) \\
& \times (t \D_{\hat{i}} \D_{k} \varphi_{\hat{T},2}(z) + (1-t) \D_{\hat{i}} \D_{k} \varphi_{\hat{T},1}(z))
\end{aligned} \right. & (d) \\
+ (D_{n}H)(Y(t,z),P(t,z)) & (e)
\end{array}
\end{array}
\end{equation}
for $z \in \hat{\Omega}.$ Consider the function
$$\mathcal{P}(z) = \frac{1}{|\bop_{n-2}(z)|^{\alpha}} \text{ for } z \in \hat{\Omega} \text{ with } \bop_{n-2}(z) \neq 0.$$
We will show there is $\CC>0$ depending on $n,$ the $C^{2}$ norm of the $G^{\hat{i} \hat{j}},$ $\epsilon>0,$ and the Lipschitz constant of $H$ so that $|d(z)| \leq \CC \mathcal{P}(z)$ for all $z \in \hat{\Omega}$ with $\bop_{n-2}(z) \neq 0.$ We consider each term (a)-(e) labelled above as follows:
\begin{itemize}
 \item[(a)] Note that $G^{\hat{i} \hat{j}} \in C^{\infty}(B^{n}_{8}(0) \times \R^{n-1})$ for each $\hat{i},\hat{j} \in \{ 1,\ldots,n-1 \}$ by \eqref{hrtc17}. Also recall $Y(t,z) \in B^{n}_{5}(0)$ and $P(t,z) \in B^{n-1}_{\epsilon}(0)$ by \eqref{hrtc26} for each $t \in [0,1]$ and $z \in \hat{\Omega}.$ We conclude there is $\CC>0$ depending on $n$ and
 $$\max_{\hat{i},\hat{j} \in \{ 1,\ldots, n-1 \}} \| \D_{\hat{i}}D_{n} G^{\hat{i} \hat{j}} \|_{C^{2}(B^{n}_{5}(0) \times B^{n-1}_{\epsilon}(0))}$$
 so that for each $t \in [0,1]$ and $z \in \hat{\Omega}$ with $\bop_{n-2}(z) \neq 0$
 $$\left| \sum_{\hat{i},\hat{j}=1}^{n-1} (\D_{\hat{i}}D_{n} G^{\hat{i} \hat{j}})(Y(t,z),P(t,z)) P_{\hat{j}}(t,z) \right| \leq \CC \leq \CC \mathcal{P}(z);$$
 we as well used $\bop_{n-2}(z) \in B^{n-2}_{1}(0)$ for each $z \in \hat{\Omega}$ and the definition of $P_{\hat{j}}$ given in \eqref{hrtc25}.
 \item[(b)] We argue exactly as in (a), and conclude by \eqref{hrtc17},\eqref{hrtc19},\eqref{hrtc26} that for each $t \in [0,1]$ and $z \in \hat{\Omega}$ with $\bop_{n-2}(z) \neq 0$
 $$\left| \sum_{\hat{i},\hat{j}=1}^{n-1} (D^{2}_{n}G^{\hat{i} \hat{j}})(Y(t,z),P(t,z)) P_{\hat{i}}(t,z) P_{\hat{j}}(t,z) \right| \leq \CC \leq \CC \mathcal{P}(z)$$
 where $\CC>0$ depends on $n$ and $\max_{\hat{i},\hat{j} \in \{ 1,\ldots, n-1 \}} \| D_{n}^{2} G^{\hat{i} \hat{j}} \|_{C^{2}(B^{n}_{5}(0) \times B^{n-1}_{\epsilon}(0))}.$
 \item[(c)]  Using \eqref{hrtc9},\eqref{hrtc19},\eqref{hrtc21},\eqref{hrtc25},\eqref{hrtc26} and \eqref{app5} from Lemma \ref{appendixlemma} we compute for $t \in [0,1]$ and $z \in \hat{\Omega}$ with $\bop_{n-2}(z) \neq 0$
 $$\begin{aligned}
\Big| \sum_{\hat{i},\hat{j}=1}^{n-1} (D_{n} & G^{\hat{i} \hat{j}})(Y(t,z),P(t,z)) \left( t \D_{\hat{i}} \D_{\hat{j}} \varphi_{\hat{T},2}(z) + (1-t) \D_{\hat{i}} \D_{\hat{j}} \varphi_{\hat{T},1}(z) \right) \Big| \\
& \leq (n-1)^{2} |Y(t,z)| \left( \frac{\CC}{z_{n-1}} \right) \leq \CC (n-1)^{2} \left( \frac{(1+\epsilon)|z|}{z_{n-1}} \right) \\
& \leq (1+\epsilon) \CC (n-1)^{2} \left( \frac{|\bop_{n-2}(z)| + z_{n-1}}{z_{n-1}} \right) \\
& = (1+\epsilon) \CC (n-1)^{2} \left( \frac{|\bop_{n-2}(z)|}{z_{n-1}} + 1 \right) \\
& \leq (1+\epsilon) \CC (n-1)^{2} \left( \frac{1}{2 \epsilon |\bop_{n-2}(z)|^{\alpha}} + \mathcal{P}(z) \right) \\
& = (1+\epsilon) \CC (n-1)^{2} \left( \frac{1}{2\epsilon}+1 \right) \mathcal{P}(z)
\end{aligned}$$
where $\CC=\CC(n)>0$ is as in \eqref{hrtc21}.
 \item[(d)] We compute as in (c) using \eqref{hrtc9},\eqref{hrtc19},\eqref{hrtc21},\eqref{hrtc25},\eqref{hrtc26} and \eqref{app5} from Lemma \ref{appendixlemma} to conclude that for $t \in [0,1]$ and $z \in \hat{\Omega}$ with $\bop_{n-2}(z) \neq 0$
 $$\begin{aligned}
 & \left| \sum_{\hat{i},\hat{j},k=1}^{n-1} \left\{ \begin{aligned}
 \left( \frac{\partial D_{n}G^{\hat{i} \hat{j}}}{\partial p_{k}} \right)(Y(t,z),P(t,z)) P_{\hat{j}}(t,z) \\
 \times (t \D_{\hat{i}} \D_{k} \varphi_{\hat{T},2} + (1-t) \D_{\hat{i}} \D_{k} \varphi_{\hat{T},1}) 
 \end{aligned} \right. \right| \\
 & \ \ \leq \epsilon \CC (n-1)^{3} \frac{|Y(t,z)|}{z_{n-1}} \leq (1+\epsilon) \CC (n-1)^{3} \left( \frac{1}{2}+\epsilon \right) \mathcal{P}(z)
 \end{aligned}$$
 where $\CC = \CC(n)>0$ is as in \eqref{hrtc21}. 
 \item[(e)] We estimate using \eqref{hrtc17},\eqref{hrtc19},\eqref{hrtc26} for $t \in [0,1]$ and $z \in \hat{\Omega}$ with $\bop_{n-2}(z) \neq 0$
 $$|(D_{n}H)(Y(t,z),P(t,z))| \leq \CC \leq \CC \mathcal{P}(z)$$
 where $\CC>0$ depends on the Lipschitz constant of $H.$
\end{itemize} 
These five calculations show as claimed that there is a $\CC>0$ depending only on $n,$ 
$$\max_{\hat{i},\hat{j} \in \{ 1,\ldots, n-1 \}} \| \D_{\hat{i}}D_{n} G^{\hat{i} \hat{j}} \|_{C^{2}(B^{n}_{5}(0) \times B^{n-1}_{\epsilon}(0))}, \ \max_{\hat{i},\hat{j} \in \{ 1,\ldots, n-1 \}} \| D_{n}^{2} G^{\hat{i} \hat{j}} \|_{C^{2}(B^{n}_{5}(0) \times B^{n-1}_{\epsilon}(0))},$$
$\epsilon>0,$ and the Lipschitz constant of $H$ (see \eqref{hrtc17}) so that
$$|d(z)| \leq \CC \mathcal{P}(z) = \frac{\CC}{|\bop_{n-2}(z)|^{\alpha}} \text{ for } z \in \hat{\Omega} \text{ with } \bop_{n-2}(z) \neq 0.$$
However, since $\alpha \in \left( 0,\frac{n-2}{n-1} \right)$ by \eqref{hrtc2}, then 
\begin{equation} \label{hrtc31}
d \in L^{q}(\hat{\Omega}) \text{ for each } q \in \left( n-1,\frac{n-2}{\alpha} \right)
\end{equation}
as claimed. We can now apply Lemma \ref{appendixhopflemma}.
       
\bigskip

Lastly, we can thus compute using \eqref{hrtc11},\eqref{hrtc15},\eqref{hrtc16},\eqref{hrtc17},\eqref{hrtc25},\eqref{hrtc26} and \eqref{app5} that
$$\| a^{ij} \|_{C(\hat{\Omega})} \leq \CC \text{ for each } i,j \in \{ 1,\ldots,n-1 \}$$
where $\CC=\CC(n)>0.$ This together with 
$$\text{\eqref{hrtc19},\eqref{hrtc22},\eqref{hrtc23},\eqref{hrtc24},\eqref{hrtc27},\eqref{hrtc28},\eqref{hrtc29},\eqref{hrtc31}}$$
contradict Lemma \ref{appendixhopflemma}, if $\epsilon=\epsilon(n)>0$ is chosen sufficiently small in \eqref{hrtc3}. We conclude that for each half-regular $\x \in \spt \partial T \cap B_{\rho}(x),$ every tangent cone of $T$ at $\x$ must be a sum of half-hyperplanes with constant orientation after rotation.
\end{proof}

\bigskip

Having shown Lemma \ref{halfregulartangentcones}, then proving our main result Theorem \ref{main} follows proving Theorem 5.3 of \cite{R18}. As in \cite{R18}, it is convenient and of interest to first prove the following theorem, which is a generalization of Theorem 5.2 of \cite{R18}.

\begin{theorem} \label{hyperplanetangentconeregularity} Suppose $\alpha \in (0,1],$ $U \subseteq \R^{n+1}$ is an open set, and suppose $T \in \TI^{1,\alpha}_{n,loc}(U)$ where $\partial T$ has Lipschitz co-oriented mean curvature. If $x \in \spt \partial T$ and $T$ at $x$ has tangent cone which is a hyperplane with constant orientation but non-constant multiplicity (as in Definition \ref{conesdefinition}), then $x \in \reg \partial T.$ \end{theorem}

\begin{proof}
The proof is exactly as the proof of Theorem 5.2 of \cite{R18}. Since the proof is brief, we include it here for convenience.

\bigskip

Suppose for contradiction $x \in \sing \partial T$ and that $T$ at $x$ has tangent cone which is a hyperplane with constant orientation but non-constant multiplicity. Theorem \ref{hyperplanetangentconegraph} implies there is $\rho \in (0,\dist(x,\partial U))$ so that $T$ at every $\hat{x} \in B_{\rho}(x) \cap \spt \partial T$ has unique tangent cone which is a hyperplane with constant orientation but non-constant multiplicity. However, by Lemma \ref{halfregular} we can find a half-regular $\x \in B_{\rho}(x) \cap \sing \partial T.$ This contradicts Lemma \ref{halfregulartangentcones}.
\end{proof}

\bigskip

We are now ready to state our main result.

\begin{theorem} \label{main} Suppose $\alpha \in (0,1],$ $U \subseteq \R^{n+1}$ is an open set, and suppose $T \in \TI^{1,\alpha}_{n,loc}(U)$ where $\partial T$ has Lipschitz co-oriented mean curvature. Suppose $x \in \spt \partial T$ and that there exists $\rho \in (0,\dist(x,\partial U))$ and $C^{1}$ hypersurfaces-with-boundary $M_{1},\ldots,M_{A}$ in $B_{\rho}(x)$ so that either:
\begin{enumerate}
 \item[(1)] $\spt T \cap B_{\rho}(x) = \bigcup_{a=1}^{A} (\clos M_{a}) \cap B_{\rho}(x),$ or
 \item[(2)] $\spt T \cap B_{\rho}(x) \subseteq \bigcup_{a=1}^{A} (\clos M_{a}) \cap B_{\rho}(x)$ and $T^{\perp}_{x} \partial T \not \subset T_{x} M_{a}$ for each $a \in \{ 1,\ldots,A \}$ such that $x \in \clos M_{a}.$
\end{enumerate}
Then there is $\sigma \in (0,\rho)$ and $\mathcal{B} \in \{ 1,\ldots,2 \Theta_{T}(x) \}$ so that
$$\spt T \cap B_{\sigma}(x) = \bigcup_{b=1}^{\mathcal{B}} (\clos W_{b}) \cap B_{\sigma}(x)$$
for orientable $C^{1,\alpha}$ hypersurfaces-with-boundary $W_{1},\ldots,W_{\mathcal{B}}$ in $B_{\sigma}(x).$ For each $b \in \{ 1,\ldots,\mathcal{B} \}$ we have $x \in \partial W_{b}$ and $W_{b} \cap \spt \partial T \subset \reg \partial T.$ Furthermore, for each $b,\tilde{b} \in \{ 1,\ldots,\mathcal{B} \}$ we have 
$$(\clos W_{b}) \cap (\clos W_{\tilde{b}}) \cap B_{\sigma}(x) \subseteq (\partial W_{b}) \cap (\partial W_{\tilde{b}}) \cap B_{\sigma}(x).$$
\end{theorem}

\begin{proof}
The proof is virtually word-for-word the same as the proof of Theorem 5.3 of \cite{R18}. Assuming (after translation) that $x=0 \in \spt \partial T$ and (after rotation) $T_{0} \partial T = \R^{n-1},$ then the proof here as in Theorem 5.3 of \cite{R18} is geometric, by considering slices of $T$ with respect to two-dimensional planes along perpendicularly to $\R^{n-1}.$ To complete the proof, we only need to change references to two ingredients which are different here. First, the proof of Theorem 5.3 of \cite{R18} applies Theorem 5.2 of \cite{R18} at several points; we simply replace these references with the generalization Theorem \ref{hyperplanetangentconeregularity} above. Second, the second-to-the-last paragraph of the proof of Theorem 5.3 of \cite{R18} applies the boundary regularity for stationary varifolds with $C^{1,1}$ boundary of \cite{A75}. At this instance, we must instead apply the boundary regularity for stationary varifolds with $C^{1,\alpha}$ boundary of \cite{B15}.

\bigskip

All other calculations, in particular all other references to facts from \cite{R18} itself, apply to $T \in \TI^{1,\alpha}_{n,loc}(U)$ with $\alpha \in (0,1].$ For example, the proof of Theorem 5.3 of \cite{R18} begins by referencing Lemma 3.9 of \cite{R18}, which holds for all $\alpha \in (0,1].$ 
\end{proof}

\appendix

\section{Appendix A}

We present here two lemmas, for the sake of making the proof of Lemma \ref{halfregulartangentcones} more readable. The first, Lemma \ref{appendixlemma}, is a set of calculations we use in order to apply Lemma \ref{appendixhopflemma}. The second, Lemma \ref{appendixhopflemma}, is a version of the general Hopf boundary point lemma from \cite{R17}, which we include here for convenience.

\bigskip

To state Lemma \ref{appendixlemma}, recall the notation for various derivatives set in the ninth item of \S 2.

\begin{lemma} \label{appendixlemma}
For each $p \in \R^{n-1}$ and $q \in \R^{n}$ define the $(n-1) \times (n-1)$ matrix $g(p,q)$ and the entries of the inverse by
\begin{equation} \label{app2} 
\begin{aligned}
g(p,q) = I_{n-1} & + (1+q_{n}^{2}) pp^{T} + q_{n} \left( p \bop_{n-1}(q)^{T} + \bop_{n-1}(q)p^{T} \right) \\
& + \bop_{n-1}(q) \bop_{n-1}(q)^{T}, \\
g^{ij}(p,q) = e_{i} & \cdot g(p,q)^{-1} e_{j} \text{ for } i,j \in \{ 1,\ldots,n-1 \}
\end{aligned}
\end{equation}
where $I_{n-1}$ is the $(n-1) \times (n-1)$ identity matrix; note that $g$ is generally invertible, but we will in what follows consider $g$ with $|p|,|q|$ small.

\bigskip

There is $\epsilon = \epsilon(n)>0$ sufficiently small so that if 
\begin{equation} \label{app3}
u \in C^{\infty}(B^{n}_{8}(0)) \text{ with } Du(0)=0 \text{ and } \| u \|_{C^{2}(B^{n}_{8}(0))} < \epsilon,
\end{equation}
and if we set for $i,j \in \{ 1,\ldots,n-1 \}$
\begin{equation} \label{app4}
G^{ij}(y,p) = \sqrt{\det(g(p,Du(y)))} g^{ij}(p,Du(y)),
\end{equation}
then for each $i,j \in \{ 1,\ldots,n-1 \}$ we have
\begin{equation} \label{app5}
|(D_{n}G^{ij})(y,p)|, \left| \left( \frac{\partial D_{n}G^{ij}}{\partial p_{k}} \right)(y,p) \right| \leq |y| \text{ and } \left| \left( \frac{\partial G^{ij}}{\partial p_{k}} \right)(y,p) \right| \leq 1
\end{equation}
when $y \in B^{n}_{8}(0)$ and $p \in B^{n-1}_{\epsilon}(0).$
\end{lemma}

\begin{proof}
We begin by observing that we can choose $\epsilon=\epsilon(n)>0$ sufficiently small so that by \eqref{app2}
\begin{equation} \label{app6}
\| g(p,q) - I_{n-1} \| \leq \epsilon \text{ and } \| g(p,q)^{-1} - I_{n-1} \| \leq \epsilon
\end{equation}
whenever $p \in B^{n-1}_{\epsilon}(0)$ and $q \in B^{n}_{\epsilon}(0).$ 

\bigskip

Next, using $g(p,Du(y)) g(p,Du(y))^{-1} = I_{n-1}$ and Jacobi's equation
$$D_{n} \big( \det g(p,Du(y)) \big) = \trace \Big( g(p,Du(y))^{T} D_{n} \big( g(p,Du(y)) \big) \Big)$$
we compute
\begin{equation} \label{app7}
\begin{aligned}
(D_{n} & G^{ij})(y,p) \\
= & \frac{\trace \Big( g(p,Du(y))^{T} D_{n} \big( g(p,Du(y)) \big) \Big)}{2 \sqrt{\det g(p,Du(y))}} g^{ij}(p,Du(y)) \\
& - \left\{ \begin{aligned}
& \sqrt{\det g(p,Du(y))} \\ 
& \times e_{i} \cdot g(p,Du(y))^{-1} \Big( D_{n} \big( g(p,Du(y)) \big) \Big) g(p,Du(y))^{-1} e_{j}.
\end{aligned} \right.
\end{aligned}
\end{equation}
We can thus choose $\epsilon = \epsilon(n)>0$ sufficiently small so that by \eqref{app3},\eqref{app6} 
\begin{equation} \label{app8}
|(D_{n} G^{ij})(y,p)| \leq \CC \left\| D_{n} \big( g(p,Du(y)) \big) \right\|
\end{equation}
where $\CC = \CC(n)>0,$ for each $y \in B^{n}_{8}(0)$ and $p \in B^{n-1}_{\epsilon}(0).$

\bigskip

We now compute $D_{n} \big( g(p,Du(y)) \big).$ For simplicity denote
\begin{equation} \label{app9}
\begin{aligned}
u_{n}(y) = & D_{n}u(y), & \D u(y) = & \bop_{n-1}(Du(y)), \\
u_{nn}(y) = & D^{2}_{n}u(y), & \text{and } \D u_{n}(y) = & \D D_{n}u(y),
\end{aligned}
\end{equation}
Then by \eqref{app2} we can write
$$\begin{aligned}
g(p,Du(y))) = I & + (1+ u_{n}(y)^{2}) pp^{T} + u_{n}(y) \left( p \D u(y)^{T} + \D u(y) p^{T} \right) \\
& + \D u(y) \D u(y)^{T},
\end{aligned}$$
and so
\begin{equation} \label{app10}
\begin{aligned}
D_{n} \big( g(p,Du(y)) \big) = & 2u_{n}(y) u_{nn}(y) pp^{T} + u_{nn} \left( p \D u(y)^{T} + \D u(y) p^{T} \right) \\
& + u_{n}(y) \left( p \D u_{n}(y)^{T} + \D u_{n}(y) p^{T} \right) \\
& + \D u_{n}(y) \D u(y)^{T} + \D u(y) \D u_{n}(y)^{T}
\end{aligned}
\end{equation}
whenever $p \in \R^{n-1}$ and $y \in B^{n}_{8}(0).$

\bigskip

Now to bound $(D_{n}G)(y,p),$ we see that \eqref{app3},\eqref{app9},\eqref{app10} imply if $\epsilon=\epsilon(n)>0$ is sufficiently small
\begin{equation} \label{app11}
\begin{aligned}
\left\| D_{n} \big( g(p,Du(y)) \big) \right\| \leq & \CC \| u \|_{C^{2}(B^{n}_{8}(0))} |Du(y)| \\
\leq & \CC \| u \|_{C^{2}(B^{n}_{8}(0))}^{2} |y| \leq \epsilon |y|
\end{aligned}
\end{equation}
where $\CC=\CC(n)>0,$ for each $y \in B^{n}_{8}(0)$ and $p \in B^{n-1}_{\epsilon}(0).$ This together with \eqref{app3},\eqref{app8} implies
\begin{equation} \label{app12}
|(D_{n} G^{ij})(y,p)| \leq \CC \epsilon |y| \leq |y|
\end{equation}
for each $y \in B^{n}_{8}(0)$ and $p \in B^{n-1}_{\epsilon}(0),$ if $\epsilon=\epsilon(n)>0$ is sufficiently small.

\bigskip

Next, we bound $\left| \left( \frac{\partial G^{ij}}{\partial p_{k}} \right)(y,p) \right|.$ Compute by \eqref{app2},\eqref{app9}
$$\begin{aligned}
\left( \frac{\partial g}{\partial p_{k}} \right) (p,Du(y)) = & (1+ u_{n}(y)^{2}) \left( e_{k}p^{T} + p e_{k}^{T} \right) \\
& + u_{n}(y) \left( e_{k} \D u(y)^{T} + \D u(y) e_{k}^{T} \right).
\end{aligned}$$
This together with $g(p,Du(y)) g(p,Du(y))^{-1} = I$ and \eqref{app2},\eqref{app3},\eqref{app6} gives
\begin{equation} \label{app13}
\left\| \left( \frac{\partial g}{\partial p_{k}} \right)(p,Du(y)) \right\| + \left\| \left( \frac{\partial g^{-1}}{\partial p_{k}} \right)(p,Du(y)) \right\| \leq \CC \epsilon
\end{equation}
where $\CC=\CC(n)>0,$ for each $y \in B^{n}_{8}(0)$ and $p \in B^{n-1}_{\epsilon}(0)$ if $\epsilon=\epsilon(n)>0$ is sufficiently small. We can thus compute as in \eqref{app7} using \eqref{app6},\eqref{app13}
\begin{equation} \label{app14}
\begin{aligned}
\Big| \Big( & \frac{\partial G^{ij}}{\partial p_{k}} \Big)(y,p) \Big| \\
\leq & \frac{\left| \trace \left( g(p,Du(y))^{T} \left( \frac{\partial g}{\partial p_{k}} \right)(p,Du(y)) \right) \right|}{2 \sqrt{\det g(p,Du(y))}} |g^{ij}(p,Du(y))| \\
& + \left\{ \begin{aligned}
& \sqrt{\det g(p,Du(y))} \\ 
& \times \| g(p,Du(y))^{-1} \| \left\| \left( \frac{\partial g}{\partial p_{k}} \right)(p,Du(y)) \right\| \| g(p,Du(y))^{-1} \|
\end{aligned} \right. \\
\leq & \CC \epsilon \leq 1
\end{aligned}
\end{equation}
where $\CC=\CC(n)>0,$ for each $y \in B^{n}_{8}(0)$ and $p \in B^{n-1}_{\epsilon}(0)$ if $\epsilon=\epsilon(n)>0$ is sufficiently small.

\bigskip

Lastly, let us bound $\left| \left( \frac{\partial D_{n}G^{ij}}{\partial p_{k}} \right)(y,p) \right|.$ Compute by \eqref{app9},\eqref{app10}
$$\begin{aligned}
\frac{\partial}{\partial p_{k}} \Big( D_{n} \big( g(p,Du(y)) \big) \Big) = & 2u_{n}(y) u_{nn}(y) \left( e_{k}p^{T} + p e_{k}^{T} \right) \\
& + u_{nn}(y) \left( e_{k} \D u(y)^{T} + \D u(y) e_{k}^{T} \right) \\
& + u_{n}(y) \left( e_{k} \D u_{n}(y)^{T} + \D u_{n}(y) e_{k}^{T} \right).
\end{aligned}$$
We conclude by \eqref{app3} if $\epsilon=\epsilon(n)>0$ is sufficiently small
\begin{equation} \label{app15}
\begin{aligned}
\left\| \frac{\partial}{\partial p_{k}} \Big( D_{n} \big( g(p,Du(y)) \big) \Big)  \right\| \leq & \CC \| u \|_{C^{2}(B^{n}_{8}(0))} |Du(y)| \\
\leq & \CC \| u \|_{C^{2}(B^{n}_{8}(0))}^{2} |y| \\
\leq & \epsilon |y|
\end{aligned}
\end{equation}
where $\CC=\CC(n)>0,$ for each $y \in B^{n}_{8}(0)$ and $p \in B^{n-1}_{\epsilon}(0).$ Differentiating the identity \eqref{app7} with respect to the $p_{k}$-variable, using again Jacobi's equation, gives that $\left( \frac{\partial D_{n}G^{ij}}{\partial p_{k}} \right)(y,p)$ is the sum of the terms

\begin{itemize}
 \item $\frac{\trace \left( \left( \left( \frac{\partial g}{\partial p_{k}} \right)(p,Du(y)) \right)^{T} D_{n} \big( g(p,Du(y))  \big) \right)}{2 \sqrt{\det g(p,Du(y))}} g^{ij}(p,Du(y)),$
 \item $\frac{\trace \left( g(p,Du(y))^{T} \frac{\partial}{\partial p_{k}} \Big( D_{n} \big( g(p,Du(y)) \big) \Big) \right)}{2 \sqrt{\det g(p,Du(y))}} g^{ij}(p,Du(y)),$
 \item $- \left\{ \begin{aligned}
& \frac{\trace \left( g(p,Du(y))^{T} D_{n} \big( g(p,Du(y)) \big) \right)}{4 (\det g(p,Du(y)))^{3/2}} \\
& \times \trace \left( g(p,Du(y))^{T} \left( \frac{\partial g}{\partial p_{k}} \right)(p,Du(y)) \right) g^{ij}(p,Du(y)),
\end{aligned} \right.$
 \item $\frac{\trace \left( g(p,Du(y))^{T} D_{n} \big( g(p,Du(y)) \big) \right)}{2 \sqrt{\det g(p,Du(y))}} \left( \frac{\partial g^{ij}}{\partial p_{k}} \right) (p,Du(y)),$
 \item $- \left\{ \begin{aligned}
& \frac{\trace \left( g(p,Du(y))^{T} \left( \frac{\partial g}{\partial p_{k}} \right)(p,Du(y)) \right)}{2 \sqrt{\det g(p,Du(y))}} \\ 
& \times e_{i} \cdot g(p,Du(y))^{-1} \Big( D_{n} \big( g(p,Du(y)) \big) \Big) g(p,Du(y))^{-1} e_{j},
\end{aligned} \right.$
 \item $- \left\{ \begin{aligned}
& \sqrt{\det g(p,Du(y))} \\ 
& \times e_{i} \cdot \frac{\partial}{\partial p_{k}} \big( g(p,Du(y))^{-1} \big) \Big( D_{n} \big( g(p,Du(y)) \big) \Big) g(p,Du(y))^{-1} e_{j},
\end{aligned} \right.$
 \item $- \left\{ \begin{aligned}
& \sqrt{\det g(p,Du(y))} \\ 
& \times e_{i} \cdot g(p,Du(y))^{-1} \left( \frac{\partial}{\partial p_{k}} \Big( D_{n} \big( g(p,Du(y)) \big) \Big) \right) g(p,Du(y))^{-1} e_{j},
\end{aligned} \right.$ 
 \item $- \left\{ \begin{aligned}
& \sqrt{\det g(p,Du(y))} \\ 
& \times e_{i} \cdot g(p,Du(y))^{-1} \Big( D_{n} \big( g(p,Du(y)) \big) \left( \frac{\partial g^{-1}}{\partial p_{k}} \right)(p,Du(y)) e_{j}.
\end{aligned} \right.$
\end{itemize}
Observe that each term contains either
$$D_{n} \big( g(p,Du(y)) \big) \text{ or } \frac{\partial}{\partial p_{k}} \Big( D_{n} \big( g(p,Du(y)) \big) \Big).$$
We can thus compute using \eqref{app3},\eqref{app6},\eqref{app11},\eqref{app13},\eqref{app15}
$$\left| \left( \frac{\partial D_{n}G^{ij}}{\partial p_{k}} \right)(y,p) \right| \leq \CC \epsilon |y| \leq |y|$$
where $\CC=\CC(n)>0,$ for $y \in B^{n}_{8}(0)$ and $p \in B^{n-1}_{\epsilon}(0)$ if $\epsilon=\epsilon(n)>0$ is sufficiently small. This together with \eqref{app12},\eqref{app14} gives \eqref{app5}.
\end{proof}

Finally, we need a more general Hopf boundary point lemma than the usual one (found for example in Lemma 3.4 of \cite{GT83}). For this, we refer to \cite{R17}. However, we won't need as general of a result as found in Theorem 4.1 of \cite{R17}, and so for convenience we state here the more precise Hopf boundary point lemma which we presently need.

\begin{lemma} \label{appendixhopflemma}
Suppose $\hat{w} \in C^{1,\alpha}(B^{n-2}_{1}(0))$ satisfies $\hat{w}(0)=0$ and $\uD \hat{w}(0)=0,$
and let
$$\hat{\Omega} = \{ z \in B^{n-2}_{1}(0) \times (0,3): z_{n-1} > \hat{w}(\bop_{n-2}(z)) \}.$$
Also suppose $q>n-1,$ and with $\alpha \in (0,1)$ suppose
\begin{itemize}
 \item $a^{ij} \in C^{0,\alpha}(\clos \hat{\Omega}),$ $c^{i} \in L^{\infty}(\hat{\Omega})$ for $i,j \in \{ 1,\ldots,n-1 \},$ and $d \in L^{q}(\hat{\Omega}),$
 \item $\sum_{i,j=1}^{n-1} a^{ij}(z)p_{i}p_{j} \geq \frac{1}{2} |p|^{2}$ for all $p \in \R^{n-1}$ and $z \in \hat{\Omega},$
 \item $a^{ij}(0) = a^{ji}(0)$ for each $i,j \in \{ 1,\ldots,n-1 \}.$
\end{itemize}
There is an $\epsilon>0$ depending on $n$ and $\max_{i,j \in \{ 1,\ldots,n-1 \}} \| a^{ij} \|_{C(\clos \hat{\Omega})}$ such that if $\| \uD \hat{w} \|_{C(B^{n-2}_{1}(0))} < \epsilon,$ and if $s \in C^{1,\alpha}(\clos \hat{\Omega})$ is a weak solution over $\hat{\Omega}$ to the equation
$$\sum_{i,j=1}^{n-1} \D_{i} \left( a^{ij} \D_{j}s \right) + \sum_{i=1}^{n-1} c^{i} \D_{i}s + ds = 0$$
with $s(z)>s(0)=0$ for all $z \in \hat{\Omega},$ then $\D s(0) \neq 0.$
\end{lemma}

\begin{proof}
We merely sketch the proof, and refer to \cite{R17}. Define the function $\hat{s} \in C^{1,\alpha}(B^{n-1}_{1}(e_{n-1}))$ by
$$\hat{s}(z) = s(z+\hat{w}(\bop_{n-2}(z)) e_{n-1}) \text{ for } z \in B^{n-1}_{1}(e_{n-1});$$
note since $\hat{w}(0)=0$ and $\| \uD \hat{w} \|_{C(B^{n-2}_{1}(0))} < \epsilon,$ then we can choose $\epsilon=\epsilon(n)>0$ sufficiently small so that $\{ z+\hat{w}(\bop_{n-2}(z)) e_{n-1}: z \in B^{n-1}_{1}(e_{n-1}) \} \subset \hat{\Omega}.$ 

\bigskip

Consider the map 
$$\hat{W}(z) = z+\hat{w}(\bop_{n-2}(z)) e_{n-1} \text{ for } z \in B^{n-1}_{1}(e_{n-1}),$$
and the calculation
$$\begin{aligned}
\D_{i}\hat{s}(z) = & (\D_{i}s)(\hat{W}(z)) + \D_{i} \big( \hat{w}(\bop_{n-2}(z)) \big) (\D_{n-1}s)(\hat{W}(z)) \\
= & (\D_{i}s)(\hat{W}(z)) + \D_{i} \big( \hat{w}(\bop_{n-2}(z)) \big) \D_{n-1}\hat{s}(z).
\end{aligned}$$
Then defining the function
$$\hat{d}(z) = \max\{ 0,d(\hat{W}(z)) \} \text{ for } z \in B^{n-1}_{1}(e_{n-1})$$
and using $\hat{W}$ as a change of variables, we can show $\hat{s} \in C^{1,\alpha}(B^{n-1}_{1}(e_{n-1}))$ is a weak solution over $B^{n-1}_{1}(e_{n-1})$ to an equation of the form
$$\sum_{i,j=1}^{n-1} \D_{i} \left( \hat{a}^{ij} \D_{j} \hat{s} \right) + \sum_{i=1}^{n-1} \hat{c}^{i} \D_{i} \hat{s} + \hat{d} \hat{s} \leq 0,$$
where the coefficients satisfy
\begin{itemize}
 \item $\hat{a}^{ij} \in C^{0,\alpha}(\clos B^{n-1}_{1}(e_{n-1})),$ $\hat{c}^{i} \in L^{\infty}(B^{n-1}_{1}(e_{n-1}))$ for $i,j \in \{ 1,\ldots,n-1 \},$ and $\hat{d} \in L^{q}(B^{n-1}_{1}(e_{n-1})),$ since $\hat{w} \in C^{1,\alpha}(\clos \hat{\Omega});$
 \item $\sum_{i,j=1}^{n-1} \hat{a}^{ij}(z)p_{i}p_{j} \geq \frac{1}{4} |p|^{2}$ for all $p \in \R^{n-1}$ and $z \in \hat{\Omega},$ using $\| \uD \hat{w} \|_{C(B^{n-2}_{1}(0))} < \epsilon$ with $\epsilon>0$ sufficiently small depending on $n$ and $\max_{i,j \in \{ 1,\ldots,n-1 \}} \| a^{ij} \|_{C(\clos \hat{\Omega})};$
 \item $\hat{d}$ is weakly non-positive over $B^{n-1}_{1}(e_{n-1})$ (see Definition 2.5 of \cite{R17});
 \item $\hat{a}^{ij}(0) = \hat{a}^{ji}(0)$ for each $i,j \in \{ 1,\ldots,n-1 \},$ since $\hat{w}(0)=0$ and $\uD \hat{w}(0)=0.$
\end{itemize}
Moreover, we still have $\hat{s}(z) > \hat{s}(0)=s(0)=0,$ since $\hat{w}(0)=0.$ We can thus apply Lemma 3.3 of \cite{R17} (with $n-1$ in place of $n$), using as well Remarks 2.2,4.2(i) of \cite{R17}, to conclude
$$0 \neq \D \hat{s}(0) = \D(s(z+\hat{w}(\bop_{n-2}(z))))|_{z=0} = \D s(0),$$
using again $\hat{w}(0)=0$ and $\uD \hat{w}(0)=0.$
\end{proof}

\end{flushleft}
\end{document}